\documentclass[a4paper,10pt]{amsart}						

	\usepackage{amsthm, amsfonts, verbatim,appendix, graphicx, subfig, psfrag}
	\usepackage{amsmath}   								
	\usepackage{mathrsfs}			
	\usepackage{hyperref}								

	\numberwithin{equation}{section}
	\numberwithin{figure}{section}

	\newtheorem{thm}{Theorem}
  	
  	\newtheorem{lem}[thm]{Lemma}
  	\newtheorem{prop}[thm]{Proposition}

	\theoremstyle{definition}

  	\newtheorem{rem}[thm]{Remark}

	\newcommand{\M}{\mathcal{M}}
	\newcommand{\Mbar}{\overline{\mathcal{M}}}

	\newcommand{\g}{\mathfrak{g}}

	\newcommand{\Li}{\mathscr{L}}

	\newcommand{\PP}{\mathbb{P}}
	
	\newcommand{\QQ}{\mathbb{Q}}
	
	\makeatletter
	\@namedef{subjclassname@2010}{%
	\textup{2010} Mathematics Subject Classification}
	\makeatother

\begin{document}

\title{Brill-Noether loci in codimension two}
\author{Nicola Tarasca}
\email{tarasca@math.uni-hannover.de}
\address{Institut f\"ur Algebraische Geometrie, Leibniz Universit\"at Hannover, Welfengarten~1, 30167 Hannover}

\begin{abstract}
Let us consider the locus in the moduli space of curves of genus $2k$
defined by curves with a pencil of degree $k$. Since the Brill-Noether
number is equal to $-2$, such a locus has codimension two. Using the method of test surfaces, we 
compute the class of its closure in the moduli space of stable
curves.
\end{abstract}

\maketitle

The classical Brill-Noether theory is of crucial importance for the geometry of moduli of curves. While a general curve admits only linear series with non-negative Brill-Noether number, the locus $\M^r_{g,d}$ of curves of genus $g$ admitting a $\g^r_d$ with negative Brill-Noether number $\rho(g,r,d):=g-(r+1)(g-d+r)<0$ is a proper subvariety of $\M_g$. Harris, Mumford and Eisenbud have extensively studied the case $\rho(g,r,d)=-1$ when $\M^r_{g,d}$ is a divisor in $\M_g$. They computed the class of its closure in $\Mbar_g$ and found that it has slope $6 + 12/(g+1)$. Since for $g\geq 24$ this is less than $13/2$ the slope of the canonical bundle, it follows that $\Mbar_g$ is of general type for $g$ composite and greater than or equal to $24$.

While in recent years classes of divisors in $\Mbar_g$ have been extensively investigated, codimension-two subvarieties are basically unexplored. A natural candidate is offered from Brill-Noether theory. Since $\rho(2k,1,k)=-2$, the locus $\M^1_{2k,k} \subset \M_{2k}$ of curves of genus $2k$ admitting a pencil of degree $k$ has codimension two (see \cite{MR1618632}). As an example, consider the hyperelliptic locus $\M^1_{4,2}$ in $\M_4$.

Our main result is the explicit computation of classes of closures of such loci. When $g\geq 12$, a basis for the codimension-two rational homology of the moduli space of stable curves $\Mbar_g$ has been found by Edidin (\cite{MR1177306}). It consists of the tautological classes $\kappa_1^2$ and $\kappa_2$ together with boundary classes. Such classes are still homologically independent for $g\geq 6$. 
Using the stability theorem for the rational cohomology of $\M_g$, Edidin's result can be extended to the case $g\geq 7$.
While there might be non-tautological generators coming from the interior of $\Mbar_g$ for $g=6$, one knows that Brill-Noether loci lie in the tautological ring of $\M_g$. 
Indeed in a similar situation, Harris and Mumford computed classes of Brill-Noether divisors in $\Mbar_g$ before knowing that Pic$_\QQ(\M_g)$ is generated solely by the class $\lambda$, by showing that such classes lie in the tautological ring of $\M_g$ (see \cite[Thm. 3]{MR664324}). Their argument works in arbitrary codimension. 

Since in our case $r=1$, in order to extend the result to the Chow group, we will use a theorem of Faber and Pandharipande, which says that classes of closures of loci of type $\M^1_{g,d}$ are tautological in $\Mbar_g$ (\cite{MR2120989}).

Having then a basis for the classes of Brill-Noether codimension-two loci, in order to determine the coefficients we use the method of test surfaces.
That is, we produce several surfaces in $\Mbar_g$ and, after evaluating the intersections on one hand with the classes in the basis and on the other hand with the Brill-Noether loci, we obtain enough independent relations in order to compute the coefficients of the sought-for classes. 

The surfaces used are bases of families of curves with several nodes, hence a good theory of degeneration of linear series is required. For this, the compactification of the Hurwitz scheme by the space of admissible covers introduced by Harris and Mumford comes into play. The intersection problems thus boil down first to counting pencils on the general curve, and then to evaluating the respective multiplicities via a local study of the compactified Hurwitz scheme.

For instance when $k=3$, we obtain the class of the closure of the trigonal locus in $\Mbar_6$.
\begin{thm}
The class of the closure of the trigonal locus in $\Mbar_6$ is 
\begin{eqnarray*}
\left[ \Mbar^1_{6,3} \right]_Q &= & \frac{41}{144} \kappa_1^2 -4 \kappa_2 + \frac{329}{144}\omega^{(2)}-\frac{2551}{144} \omega^{(3)} -\frac{1975}{144}\omega^{(4)} +\frac{77}{6}\lambda^{(3)}\\  
&&{} -\frac{13}{6}\lambda \delta_0 -\frac{115}{6}  \lambda \delta_1 -\frac{103}{6}\lambda \delta_2 -\frac{41}{144}\delta_0^2 -\frac{617}{144}  \delta_1^2+18 \delta_{1,1}\\
&&{} + \frac{823}{72}\delta_{1,2} +\frac{391}{72}\delta_{1,3} +\frac{3251}{360}\delta_{1,4}+\frac{1255}{72} \delta_{2,2} +\frac{1255}{72}\delta_{2,3}\\ &&{} +\delta_{0,0}+\frac{175}{72} \delta_{0,1} 
+\frac{175}{72}\delta_{0,2} -\frac{41}{72}\delta_{0,3} +\frac{803}{360}\delta_{0,4}+\frac{67}{72} \delta_{0,5}\\
&&{}+ 2\theta_1 -2 \theta_2 .
\end{eqnarray*}
\end{thm}

\noindent For all $k\geq 3$ we produce a closed formula expressing the class of $\Mbar^1_{2k,k}$.
\begin{thm}
For $k\geq 3$ the class of the locus $\Mbar^1_{2k,k}$ in $\Mbar_{2k}$ is
\begin{eqnarray*}
\left[ \Mbar^1_{2k,k} \right]_Q & = & \frac{2^{k-6}(2k-7)!!}{3(k!)} \bigg[ (3 k^2+3 k+5)\kappa_1^2 - 24k(k+5)\kappa_2 \\
			&& {}+ \sum_{i=2}^{2k-2}  \Big(-180 i^4+120 i^3 (6 k+1)-36 i^2 \left(20 k^2+24 k-5\right) \\
			&&{} +24 i \left(52 k^2-16 k-5\right)+27 k^2+123 k+5 \Big)\omega^{(i)} + \cdots \bigg].
\end{eqnarray*}
\end{thm}

The complete formula is shown in \S \ref{result}. 
We also test our result in several ways, for example by pulling-back to $\Mbar_{2,1}$.
The computations include the case $g=4$ which was previously known: the hyperelliptic locus in $\Mbar_4$ has been computed in \cite[Prop. 5]{MR2120989}.

\setcounter{tocdepth}{1}
\tableofcontents

\section{A basis for \texorpdfstring{$R^2(\Mbar_{g})$}{R2(Mbar{g})}}
\label{generatingclasses}

Let $A^*(\Mbar_g)$ be the Chow ring with $\QQ$-coefficients of the moduli space of stable curves $\Mbar_g$, and let $R^*(\Mbar_g)\subset A^*(\Mbar_g)$ be the tautological ring of $\Mbar_g$ (see \cite{MR2120989}).
In \cite{MR1177306}, Edidin gives a basis for the space of codimension-two tautological classes $R^2(\Mbar_g)$ and he shows that such a basis holds also for the codimension-two rational homology of $\Mbar_g$ for $g\geq 12$. 

Let us quickly recall the notation. There are the tautological classes $\kappa_1^2$ and $\kappa_2$ coming from the interior $\M_g$; the following products of classes from ${\rm Pic}_\QQ(\Mbar_g)$: $\lambda \delta_0, \lambda \delta_1, \lambda \delta_2, \delta_0^2$ and $\delta_1^2$; the following push-forwards $\lambda^{(i)},\lambda^{(g-i)}, \omega^{(i)}$ and $\omega^{(g-i)}$ of the classes $\lambda$ and $\omega=\psi$ respectively from $\M_{i,1}$ and $\M_{g-i,1}$ to $\Delta_i\subset \Mbar_g$: $\lambda^{(3)}, \dots, \lambda^{(g-3)}$ and $\omega^{(2)},\dots, \omega^{(g-2)}$;  for $1\leq i \leq \lfloor (g-1)/2 \rfloor$ the $Q$-class $\theta_i$ of the closure of the locus $\Theta_i$ whose general element is a union of a curve of genus $i$ and a curve of genus $g-i-1$ attached at two points; finally the classes $\delta_{ij}$ defined as follows. The class $\delta_{00}$ is the $Q$-class of the closure of the locus $\Delta_{00}$ whose general element is an irreducible curve with two nodes. For $1\leq j \leq g-1$ the class $\delta_{0j}$ is the $Q$-class of the closure of the locus $\Delta_{0j}$ whose general element is an irreducible nodal curve of geometric genus $g-j-1$ together with a tail of genus $j$. At last for $1\leq i \leq j \leq g-2$ and $i+j \leq g-1$, the class $\delta_{ij}$ is defined as $\delta_{ij}:=[\overline{\Delta}_{ij}]_Q$, where $\Delta_{ij}$ has as general element a chain of three irreducible curves with the external ones having genus $i$ and $j$. 

\begin{figure}[htbp]
\psfrag{i}[b][b]{$i$}
\psfrag{j}[b][b]{$j$}
\psfrag{g-i-1}[b][b]{$g-i-1$}
\psfrag{g-i-j}[b][b]{$g-i-j$}
\psfrag{D00}[b][b]{$\Delta_{00}$}
\psfrag{Dij}[b][b]{$\Delta_{ij}$}
\psfrag{D0j}[b][b]{$\Delta_{0j}$}
\psfrag{theta}[b][b]{$\Theta_i$}
\centering
  \includegraphics[scale=0.5]{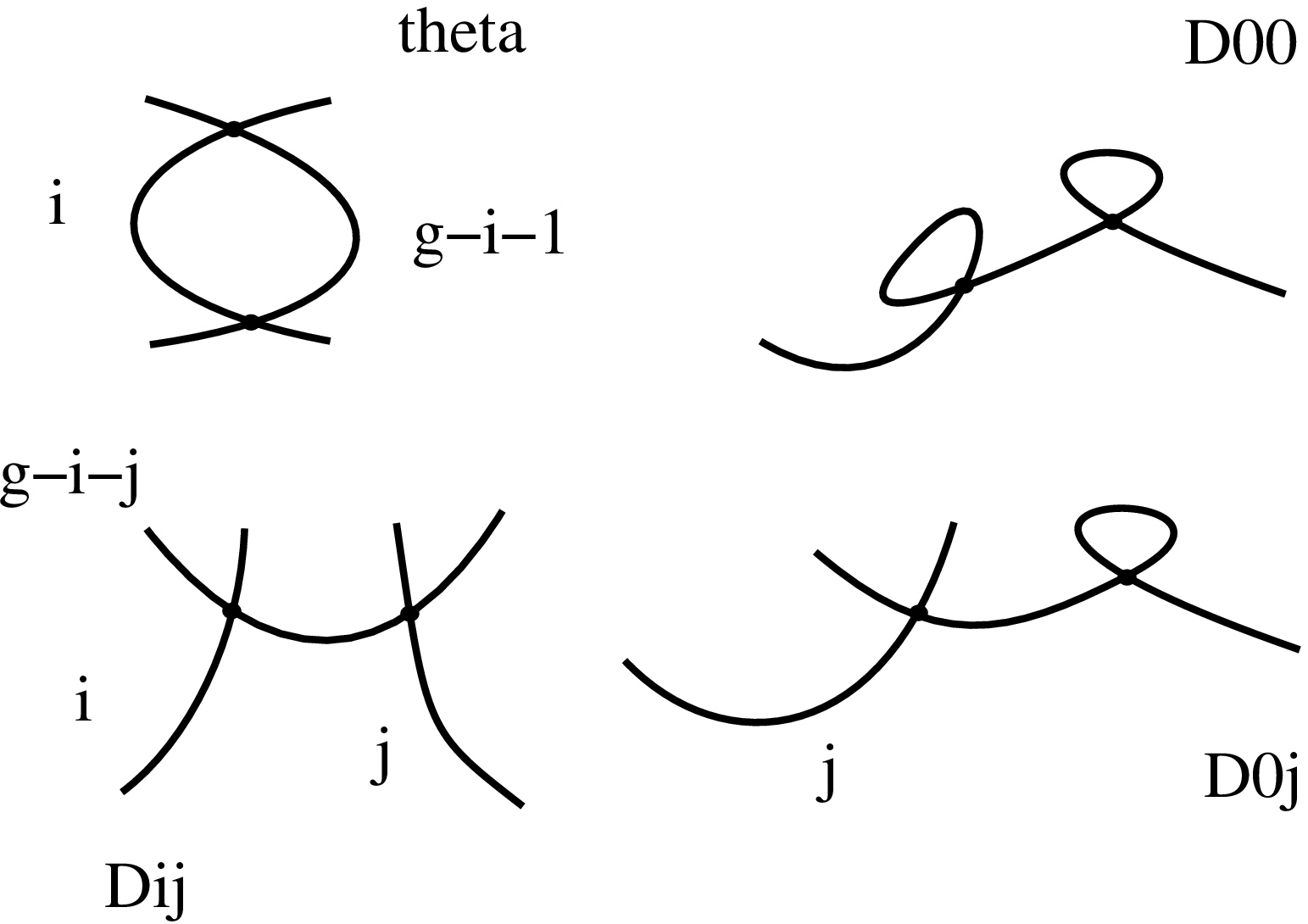}
  \caption{Loci in $\Mbar_g$}
\end{figure}

The above classes generate $R^2(\Mbar_g)$ and Edidin shows that they are homologically independent for $g\geq 6$. It follows that for $g\geq 6$ the space of codimension-two tautological classes $R^2(\Mbar_g)$ has dimension
\[
\lfloor (g^2-1)/4\rfloor + 3g-1.
\]

When $g\geq 12$, in order to conclude that the above classes form a basis also for $H_{2(3g-3)-4}(\Mbar_{g}, \QQ)$, Edidin gives an upper bound on the rank of $H_{2(3g-3)-4}(\Mbar_{g}, \QQ)$ using that $H^4(\M_g, \QQ)=\QQ^2$ for $g\geq 12$ as shown by Harer. 
By the stability theorem for the rational cohomology of $\M_g$, we know that
\[
H^k(\M_g,\QQ)\cong H^k(\M_{g+1},\QQ) \cong  H^k(\M_{g+2},\QQ) \cong \cdots
\]
for $3k\leq 2(g-1)$ (see for instance \cite{wahl}). It follows that the above classes form a basis for $H_{2(3g-3)-4}(\Mbar_{g}, \QQ)$ when $g\geq 7$.

While for $g=6$ there might be non-tautological generators coming from the interior of $\Mbar_g$, using an argument similar to \cite[Thm. 3]{MR664324} one knows that classes of Brill-Noether loci $\M^r_{g,d}$ lie in the tautological ring of $\M_g$. It follows that classes in $H_{2(3g-3)-4}(\Mbar_{g}, \QQ)$ of closures of Brill-Noether loci of codimension two can be expressed as linear combinations of the above classes for $g\geq 6$.

In the case $r=1$, we know more: classes of closures of Brill-Noether loci $\Mbar^1_{g,d}$ lie in the tautological ring of $\Mbar_g$ (see \cite[Prop. 1]{MR2120989}). Hence for $g=2k \geq 6$ we can write 
\begin{eqnarray}
\label{generalclass}
\left[ \Mbar^1_{2k,k} \right]_Q &\!\!\! = \!\!\!& A_{\kappa_1^2}\kappa_1^2 + A_{\kappa_2}\kappa_2 + A_{\delta_0^2}\delta_0^2 + A_{\lambda\delta_0}\lambda\delta_0 + A_{\delta_1^2}\delta_1^2  + A_{\lambda\delta_1}\lambda\delta_1  \nonumber \\
&&{}+ A_{\lambda\delta_2}\lambda\delta_2 + \sum_{i=2}^{g-2} A_{\omega^{(i)}}\omega^{(i)} + \sum_{i=3}^{g-3} A_{\lambda^{(i)}}\lambda^{(i)} + \sum_{i,j} A_{\delta_{ij}}\delta_{ij} \\
&&{}+ \sum_{i=1}^{\lfloor (g-1)/2 \rfloor} A_{\theta_i}\theta_{i}  \nonumber
\end{eqnarray}
in $R^2(\Mbar_{g}, \QQ)$, for some rational coefficients $A$.

\section{On the method of test surfaces}
\label{kappa&omega}

The method of test surfaces has been developed in \cite{MR1177306}. See \cite[\S 3.1.2, \S 3.4 and Lemma 4.3]{MR1177306} for computing the restriction of the generating classes to cycles parametrizing curves with nodes. In this section we summarize some results which will often be used in \S\ref{testsurfaces}.

In order to compute the restriction of $\kappa_1^2$ to test surfaces, we will use Mumford's formula for $\kappa_1$: if $g>1$ then $\kappa_1 = 12 \lambda -\delta$ in Pic$_\QQ(\Mbar_{g})$ (see \cite{MR0450272}). 
In the following proposition we note how to compute the restriction of the class $\kappa_2$ and the classes $\omega^{(i)}$ and $\lambda^{(i)}$ to a certain kind of surfaces which will appear in \S\ref{testsurfaces} in (S1)-(S14).

\begin{prop}
Let $\pi_1\colon X_1\rightarrow B_1$ be a one-dimensional family of stable curves of genus $i$ with section $\sigma_1\colon B_1\rightarrow X_1$ and similarly let $\pi_2\colon X_2\rightarrow B_2$ be a one-dimensional family of stable curves of genus $g-i$ with section $\sigma_2\colon B_2\rightarrow X_2$. Obtain a two-dimensional family of stable curves $\pi\colon X \rightarrow B_1 \times B_2$ as the union of $X_1\times B_2$ and $B_1\times X_2$ modulo glueing $\sigma_1(B_1)\times B_2$ with $B_1 \times \sigma_2(B_2)$.
Then the class $\kappa_2$ and the classes $\omega^{(i)}$ and $\lambda^{(i)}$ restrict to $B_1 \times B_2$ as follows
\[
\begin{array}{rcll}
\kappa_2 & = & 0 &\\
\omega^{(i)} = \omega^{(g-i)} & = & -{\pi_1}_*(\sigma_1^2(B_1)){\pi_2}_*(\sigma_2^2(B_2))&\mbox{if $2\leq i <g/2$}\\
\omega^{(g/2)} & =&  -2{\pi_1}_*(\sigma_1^2(B_1)){\pi_2}_*(\sigma_2^2(B_2)) & \mbox{if $g$ is even}\\
\omega^{(j)} &=& 0 		& \mbox{for $j\not\in \{i,g-i\}$}\\
\lambda^{(i)} &=& \lambda_{B_1}{\pi_2}_*(\sigma^2_2(B_2)) & \mbox{if $3\leq i <g/2$}\\
\lambda^{(g-i)} &=& \lambda_{B_2}{\pi_1}_*(\sigma^2_1(B_1))& \mbox{idem}\\
\lambda^{(g/2)} &=& \lambda_{B_1}{\pi_2}_*(\sigma^2_2(B_2))+\lambda_{B_2}{\pi_1}_*(\sigma^2_1(B_1))&\mbox{if $g$ is even}\\
\lambda^{(j)} &=& \lambda_{B_1}\delta_{j-i,1}|_{B_2}+ \lambda_{B_2}\delta_{j-g+i,1}|_{B_1}		&\mbox{for $j\not\in \{i,g-i\}$}
\end{array}
\]
where $\delta_{h,1}|_{B_1}\in {\rm Pic}_\QQ(\Mbar_{i,1})$ and similarly  $\delta_{h,1}|_{B_2}\in {\rm Pic}_\QQ(\Mbar_{g-i,1})$.
\end{prop}

\begin{proof}
Let $\nu\colon \widetilde{X}\rightarrow X$ be the normalization, where $\widetilde{X} := X_1\times B_2 \cup B_1\times X_2$. Let $K_{X/B_1\times B_2} = c_1(\omega_{X/B_1\times B_2})$. We have
\begin{eqnarray*}
\kappa_2 &=& \pi_*\left( K^3_{X/B_1\times B_2} \right)\\
		&=& \pi_* \nu_*\left((\nu^*K_{X/B_1\times B_2}	)^3\right)
\end{eqnarray*}
where we have used that $\nu$ is a proper morphism, hence the push-forward is well-defined. 
One has
\[
K_{\widetilde{X}/B_1\times B_2} = \left( K_{X_1/B_1}\times B_2 \right)\oplus \left( B_1 \times K_{X_2/B_2}\right)
\]
hence
\[
\nu^*K_{X/B_1\times B_2} = \left( (K_{X_1/B_1}+\sigma_1(B_1))\times B_2 \right)\oplus \left( B_1 \times (K_{X_2/B_2}+\sigma_2(B_2))\right).
\]
Finally
\[
 \left( (K_{X_1/B_1}+\sigma_1(B_1))\times B_2 \right)^3=  (K_{X_1/B_1}+\sigma_1(B_1))^3\times B_2 =0
\]
since $K_{X_1/B_1}+\sigma_1(B_1)$ is a class on the surface $X_1$, and similarly for $B_1 \times (K_{X_2/B_2}+\sigma_2(B_2))$, 
hence $\kappa_2$ is zero.

The statement about the classes $\omega^{(i)}$ and $\lambda^{(i)}$ follows almost by definition. For instance, since the divisor $\delta_i$ is
\[
\delta_i = \pi_*(\sigma^2_1(B_1)\times B_2)+\pi_*(B_1\times \sigma^2_2(B_2))
\]
we have
\begin{eqnarray*}
\omega^{(i)} &=& -{\pi_1}_*(\sigma_1^2(B_1)) \cdot {\pi_2}_*(\sigma_2^2(B_2)) .
\end{eqnarray*}
The other equalities follow in a similar way. 
\end{proof}

\section{Enumerative geometry on the general curve}


Let $C$ be a complex smooth projective curve of genus $g$ and $l=(\Li, V)$ a linear series of type $\g^r_d$ on $C$, that is $\Li \in {\rm Pic}^d(C)$ and $V\subset H^0(C, \Li)$ is a subspace of vector-space dimension $r+1$. The {\it vanishing sequence} $a^l(p): 0\leq a_0 <\dots <a_r\leq d$ of $l$ at a point $p\in C$ is defined as the sequence of distinct order of vanishing of sections in $V$ at $p$, and the {\it ramification sequence} $\alpha^l(p): 0\leq \alpha_0 \leq \dots \leq \alpha_r\leq d-r$ as $\alpha_i := a_i-i$, for $i=0,\dots,r$. The {\it weight} $w^l(p)$ will be the sum of the $\alpha_i$'s.

Given an $n$-pointed curve $(C,p_1,\dots,p_n)$ of genus $g$ and $l$ a $\g^r_d$ on $C$, the {\it adjusted Brill-Noether number} is 
\begin{eqnarray*}
\rho(C,p_1,\dots p_n) &=& \rho(g,r,d,\alpha^l(p_1),\dots,\alpha^l(p_n)) \\
				  &:=& g-(r+1)(g-d+r)-\sum_{i,j} \alpha^l_j(p_i).
\end{eqnarray*}

\subsection{Fixing two general points}
\label{fixing2genpts}

Let $(C,p,q)$ be a general 2-pointed curve of genus $g\geq 1$ and let $\alpha=(\alpha_0, \dots,\alpha_r)$ and $\beta=(\beta_0,\dots,\beta_r)$ be  Schubert indices of type $r,d$ (that is $0\leq \alpha_0 \leq \dots \leq \alpha_r \leq d-r$ and similarly for $\beta$) such that $\rho(g,r,d, \alpha,\beta)=0$. 
The number of linear series $\g^r_d$ having ramification sequence $\alpha$ at the point $p$ and $\beta$ at the point $q$ is counted by the {\it adjusted Castelnuovo number}
\[
 g! \det\left( \frac{1}{[\alpha_i+i+\beta_{r-j}+r-j+g-d]!} \right)_{0\leq i,j\leq r}
\]
where $ 1/[\alpha_i+i+\beta_{r-j}+r-j+g-d]!$ is taken to be $0$ when the denominator is negative (see \cite[Proof of Prop. 2.2]{MR2574363} and \cite[Ex. 14.7.11 (v)]{MR1644323}). Note that the above expression may be zero, that is the set of desired linear series may be empty.  

When $r=1$ let us denote by $N_{g,d,\alpha,\beta}$ the above expression.
If $\alpha_0=\beta_0=0$ then
\begin{multline*}
N_{g,d,\alpha,\beta} = g! \Bigg(\frac{1}{(\beta_1+1+g-d)!(\alpha_1+1+g-d)!}\\
-\frac{1}{(g-d)!(\alpha_1+\beta_1+2+g-d)!} \Bigg).
\end{multline*}
Subtracting the base locus $\alpha_0 p + \beta_0 q$, one can reduce the count to the case $\alpha_0=\beta_0=0$, hence $N_{g,d,\alpha,\beta}=N_{g,d-\alpha_0-\beta_0,(0,\alpha_1-\alpha_0),(0,\beta_1-\beta_0)}$.

In the following we will also use the abbreviation $N_{g,d,\alpha}$ when $\beta$ is zero, that is $N_{g,d,\alpha}$ counts the linear series with the only condition of ramification sequence $\alpha$ at a single general point.

\subsection{A moving point}
\label{movingpt}

Let $C$ be a general curve of genus $g>1$ and $\alpha = (\alpha_0,\alpha_1)$ be a Schubert index of type $1,d$ (that is $0\leq \alpha_0 \leq \alpha_1\leq d-1$). When $\rho(g,1,d,\alpha)=-1$, there is a finite number $n_{g,d,\alpha}$ of $(x, l_C) \in C\times W^1_d(C)$ such that $\alpha^{l_C}(x)=\alpha$. (Necessarily $\rho(g,1,d)\geq 0$ since the curve is general.) Assuming $\alpha_0=0$, one has $\alpha_1=2d-g-1$ and
\[
n_{g,d,\alpha}=  (2d-g-1)(2d-g)(2d-g+1)\frac{g!}{d!(g-d)!}.
\]
If $\alpha_0>0$ then $n_{g,d,\alpha}=n_{g,d-\alpha_0,(0,\alpha_1-\alpha_0)}$. Each $\tilde{l}_C:=l_C(-\alpha_0 x)$ satisfies $h^0(\tilde{l}_C)=2$, is generated by global sections, and $H^0(C, \tilde{l}_C)$ gives a covering of $\PP^1$ with ordinary brach points except for a $(\alpha_1-\alpha_0)$-fold branch point, all lying over distinct points of $\PP^1$. Moreover, since for general $C$ the above points $x$ are distinct, one can suppose that fixing one of them, the $l_C$ is unique.
See \cite[Thm. B and pg. 78]{MR664324}. Clearly $\alpha$ in the lower indexes of the numbers $n$ is redundant in our notation, but for our purposes it is useful to keep track of it.

\subsection{Two moving points}
\label{m}

Let $C$ be a general curve of genus $g>1$ and $\alpha = (\alpha_0,\alpha_1)$ be a Schubert index of type $1,d$ (that is $0\leq \alpha_0 \leq \alpha_1\leq d-1$). When $\rho(g,1,d,\alpha,(0,1))=-2$ (and $\rho(g,1,d)\geq 0$), there is a finite number $m_{g,d,\alpha}$ of $(x,y,l_C) \in C \times C\times G^1_d(C)$ such that $\alpha^{l_C}(x)=\alpha$ and  $\alpha^{l_C}(y)=(0,1)$. Subtracting the base locus as usual, one can always reduce to the case $\alpha_0=0$.

\begin{lem}
Assuming $\alpha_0=0$, one has that
\[
m_{g,d,\alpha} = n_{g,d,\alpha} \cdot \left(3g-1   \right).
\]
\end{lem}

\begin{proof}
Since $\rho(g,1,d,\alpha)=-1$, one can compute first the number of points of type $x$, and then fixing one of these, use the Riemann-Hurwitz formula to find the number of points of type $y$.
\end{proof}

\section{Compactified Hurwitz scheme}

Let $H_{k,b}$ be the Hurwitz scheme parametrizing coverings $\pi\colon C\rightarrow \PP^1$ of degree $k$ with $b$ ordinary branch points and $C$ a smooth irreducible curve of genus $g$. By considering only the source curve $C$, $H_{k,b}$ admits a map to $\M_g$
\[
\sigma\colon H_{k,b} \rightarrow \M_g.
\]
In the following, we will use the compactification $\overline{H}_{k,b}$ of $H_{k,b}$ by the space of admissible covers of degree $k$, introduced by Harris and Mumford in \cite{MR664324}. 
Given a semi-stable curve $C$ of genus $g$ and a stable $b$-pointed curve $(R, p_1, p_2, \linebreak[1] \dots,  p_b)$ of genus $0$, an {\it admissible cover} is a regular map $\pi\colon C \rightarrow B$ such that the followings hold: $\pi^{-1}(B_{\rm smooth})= C_{\rm smooth}$, $\pi|_{C_{\rm smooth}}$ is simply branched over the points $p_i$ and unramified elsewhere, $\pi^{-1}(B_{\rm singular})= C_{\rm singular}$
and if $C_1$ and $C_2$ are two branches of $C$ meeting at a point $p$, then $\pi |_{C_1}$ and $\pi |_{C_2}$ have same ramification index at $p$. Note that one may attach rational tails at $C$ to cook up the degree of $\pi$.

The map $\sigma$ extends to
\[
\sigma\colon \overline{H}_{k,b} \rightarrow \Mbar_g.
\]
In our case $g=2k$, the image of this map is $\Mbar^1_{2k,k}$. It is classically known that the Hurwitz scheme is connected and its image in $\M_g$ (that is, $\M^1_{2k,k}$ in our case) is irreducible (see for instance \cite{MR0260752}).

Similarly for a Schubert index $\alpha=(\alpha_0,\alpha_1)$ of type $1,k$ such that $\rho(g,1,k,\alpha)=-1$ (and $\rho(g,1,k)\geq 0$), the Hurwitz scheme $H_{k,b}(\alpha)$ (respectively $\overline{H}_{k,b}(\alpha)$) parameterizes $k$-sheeted (admissible) coverings $\pi\colon C\rightarrow \PP^1$ with $b$ ordinary branch points $p_1,\dots,p_b$ and one point $p$ with ramification profile described by $\alpha$ (see \cite[\S 5]{MR791679}). By forgetting the covering and keeping only the pointed source curve $(C,p)$, we obtain a map $\overline{H}_{k,b}(\alpha) \rightarrow \Mbar_{g,1}$ with image the pointed Brill-Noether divisor $\Mbar^1_{g,k}(\alpha)$.

\vskip1pc

Let us see these notions at work. Let $(C,p,q)$ be a two-pointed general curve of genus $g-1\geq 1$. In the following, we consider the curve $\overline{C}$ in $\Mbar_{g,1}$ obtained identifying the point $q$ with a moving point $x$ in $C$. 
In order to construct this family of curves, one blows up $C\times C$ at $(p,p)$ and $(q,q)$ and identifies the proper transforms $S_1$ and $S_2$ of the diagonal $\Delta_C$ and $q \times C$. This is a family $\pi\colon X\rightarrow C$ with a section corresponding to the proper transform of $p\times C$, hence there exists a map $C \rightarrow \Mbar_{g,1}$. We denote by $\overline{C}$ the image of $C$ in $\Mbar_{g,1}$.

\begin{figure}[htbp]
\psfrag{P1}[b][b]{$(\PP^1)_1$}
\psfrag{P0}[b][b]{$(\PP^1)_0$}
\psfrag{p}[b][b]{$p$}
\psfrag{q}[b][b]{$q$}
\psfrag{x}[b][b]{$x$}
\psfrag{P}[b][b]{$R$}
\psfrag{E}[b][b]{$E$}
\centering
\subfloat[The case $2p\equiv q+x$]{ \includegraphics[scale=0.5]{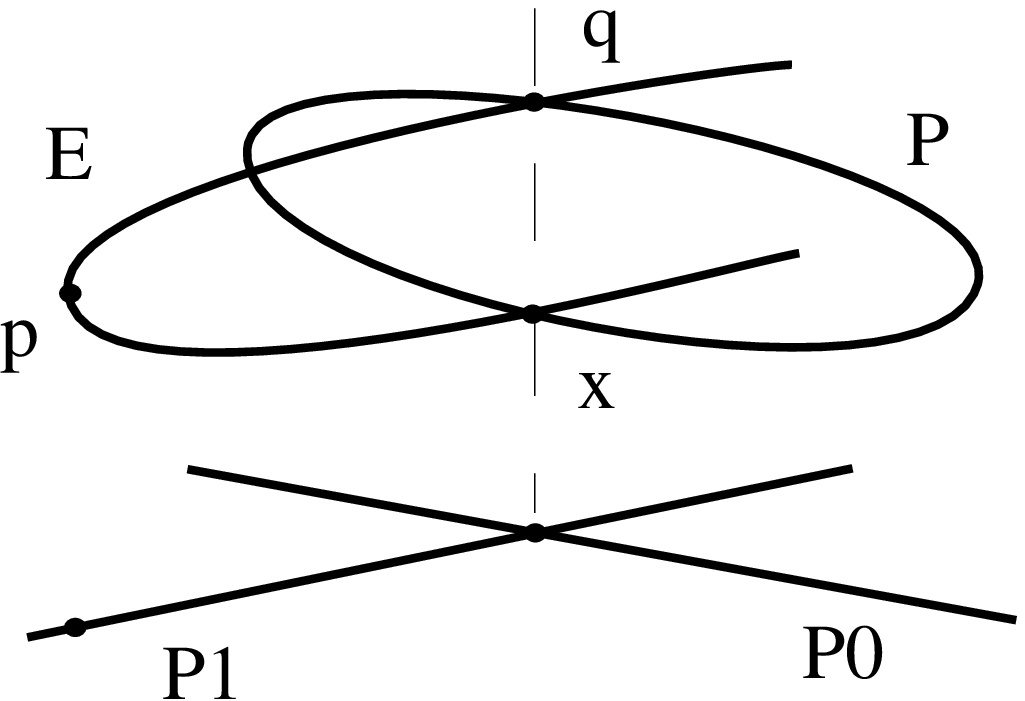}}
\subfloat[The case $x=p$]{ \includegraphics[scale=0.5]{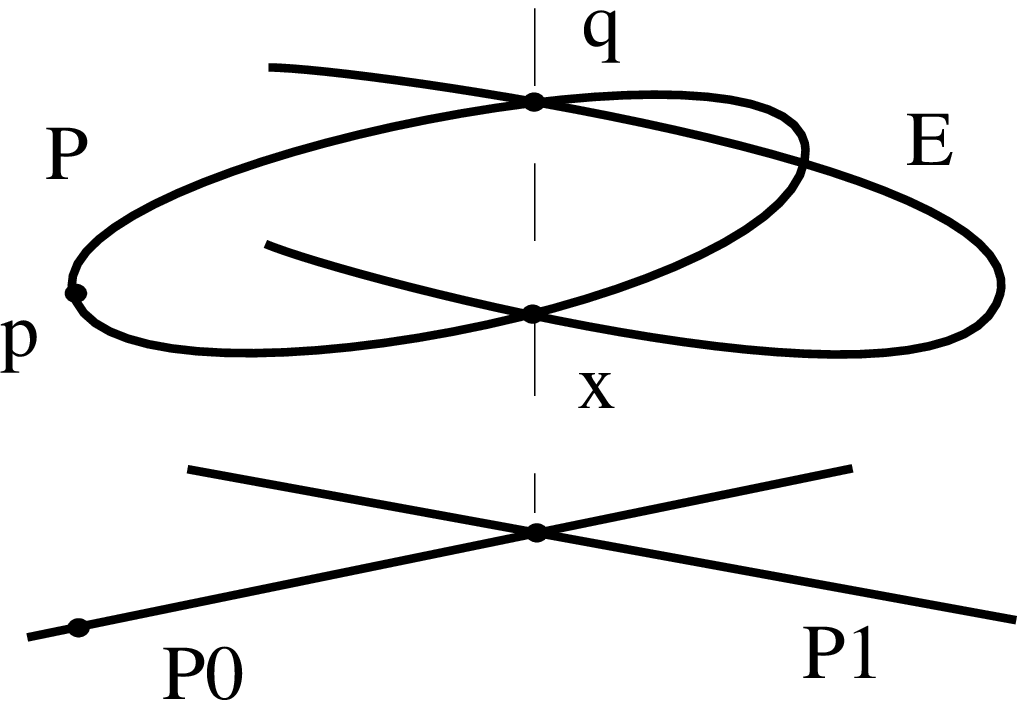}}  
  \caption{The admissible covers for the two fibers of the family $\overline{C}$ when $g=2$}
\end{figure}

\begin{lem}
\label{ex}
Let $g=2$ and let $\mathcal{W}$ be the closure of the Weierstrass divisor in $\Mbar_{2,1}$. We have that 
\[
\ell_{2,2} := \deg \left(\overline{C} \cdot \mathcal{W} \right)=2.
\] 
\end{lem}
\begin{proof}
There are two points in $\overline{C}$ with an admissible cover of degree $2$ with simple ramification at the marked point, and such admissible covers contribute with multiplicity one. Note that here $C$ is an elliptic curve. One admissible cover is for the fiber over $x$ such that $2p\equiv q+x$, and the other one for the fiber over $x=p$. In both cases the covering is determined by $|q+x|$ and there is a rational curve $R$ meeting $C$ in $q$ and $x$.

When $2p\equiv q+x$, the situation is as in \cite[Thm. 6(a)]{MR664324}. Let $C'\rightarrow P$ be the corresponding admissible covering. If\begin{eqnarray*}
\mathcal{C} \!\!\!\! & \longrightarrow & \!\!\!\! \mathcal{P} \\
&\searrow \quad \swarrow &\\
& B &
\end{eqnarray*} 
is a general deformation of $[C'\rightarrow P]$ in $\overline{H}_{2,b}(0,1)$, blowing down the curve $R$ we obtain a family of curves $\widetilde{\mathcal{C}}\rightarrow B$ with one ordinary double point. That is, $B$ meets $\Delta_0$ with multiplicity $2$. Considering the involution of $[C'\rightarrow P]$ obtained interchanging the two ramification points of $R$, we see that the map $\overline{H}_{2,b}(0,1)\rightarrow \Mbar_{2,1}$ is ramified at $[C'\rightarrow P]$. Hence $[C']$ is a transverse point of intersection of $\mathcal{W}$ with $\Delta_0$ and it follows that $\overline{C}$ and $\mathcal{W}$ meet transversally at $[C']$.

When $x=p$, the situation is similar. In a general deformation in $\overline{H}_{2,b}(0,1)$
\begin{eqnarray*}
\mathcal{C} \!\!\!\! & \longrightarrow & \!\!\!\! \mathcal{P} \\
&\searrow \quad \swarrow &\\
& B &
\end{eqnarray*} 
of the corresponding admissible covering $[C'\rightarrow P]$, one sees that $C'$ is the only fiber of $\mathcal{C}\rightarrow B$ inside $\Delta_{00}$, and at each of the two nodes of $C'$, the space $\mathcal{C}$ has local equation $x\cdot y = t$. It follows that $C'$ is a transverse point of intersection of $\mathcal{W}$ with $\Delta_{00}$. Hence $C'$ is a transverse point of intersection of $\overline{C}$ with $\mathcal{W}$.
See also \cite[\S 3]{MR735335}.
\end{proof}

\begin{lem}
\label{ell}
Let $g=2k-2>2$. The intersection of $\overline{C}$ with the pointed Brill-Noether divisor $\Mbar^1_{2k-2,k}(0,1)$ is reduced and it has degree
\[
\ell_{g,k}:= \deg \left( \overline{C}\cdot \Mbar^1_{2k-2,k}(0,1) \right)= 2\frac{(2k-3)!}{(k-2)!(k-1)!} .
\]
\end{lem}

\begin{proof}
Let us write the class of $\Mbar^1_{g,k}(0,1)$ as $a\lambda +c \psi-\sum b_i \delta_i \in {\rm Pic}_\QQ(\Mbar_{g,1})$.
First we study the intersection of the curve $\overline{C}$ with the classes generating the Picard group. Let $\pi \colon \Mbar_{g,1}\rightarrow \Mbar_g$ the map forgetting the marked point and $\sigma \colon \Mbar_g \rightarrow \Mbar_{g,1}$ the section given by the marked point. Note that on $\overline{C}$ we have
$\deg \psi = -\deg \pi_*(\sigma^2) =1$,
since the marked point is generically fixed and is blown-up in one fiber. Moreover $\deg \delta_{g-1}=1$, since only one fiber contains a disconnecting node and the family is smooth at this point.
The intersection with $\delta_0$ deserves more care. The family indeed is inside $\Delta_0$: the generic fiber has one non-disconnecting node and moreover the fiber over $x=p$ has two non-disconnecting nodes. We have to use \cite[Lemma 3.94]{MR1631825}. 
Then
\begin{eqnarray}
\quad \quad \deg \delta_0 = \deg S_1^2 + \deg S_2^2 + 1 =-2(g-1)-1+1=2-2g.
\end{eqnarray}
All other generating classes restrict to zero. Then 
\[
\deg \left( \overline{C} \cdot \left[\Mbar^1_{g,k}(0,1)\right] \right)= c + (2g-2) b_0 - b_{g-1}.
\]
On the other hand, one has an explicit expression for the class of $\Mbar^1_{g,k}(0,1)$
\[
\frac{(2k-4)!}{(k-2)!k!}\left( 6(k+1)\lambda +6(k-1)\psi -k \delta_0 + \sum_{i=1}^{g-1} 3(i+1)(2+i-2k) \delta_i \right)
\]
(see \cite[Thm. 4.5]{MR1953519}), whence the first part of the statement. 

Finally the intersection is reduced. Indeed, since the curve $C$ is general, an admissible cover with the desired property for a fiber of the family over $\overline{C}$ is determined by a unique linear series (see \cite[pg. 75]{MR664324}). Moreover, reasoning as in the proof of the previous Lemma, one sees that $\overline{C}$ and $\Mbar^1_{g,k}(0,1)$ meet always transversally. 
\end{proof}

\section{Limit linear series}
\label{lls}

The theory of limit linear series will be used. Let us quickly recall some notation and results.
On a tree-like curve, a linear series or a limit linear series is called {\it generalized} if the line bundles involved are torsion-free (see \cite[\S 1]{MR910206}).
For a tree-like curve $C=Y_1 \cup \cdots \cup Y_s$ of arithmetic genus $g$ with disconnecting nodes at the points $\{p_{ij}\}_{ij}$, let $\{l_{Y_1},\dots l_{Y_s} \}$ be a generalized limit linear series $\g^r_d$ on $C$. Let $\{q_{ik} \}_k$ be smooth points on $Y_i$, $i=1,\dots,s$. In \cite{MR846932} a moduli space of such limit series is constructed as a disjoint union of schemes on which the vanishing sequences of the aspects $l_{Y_i}$'s at the nodes are specified. A key property is the additivity of the adjusted Brill-Noether number, that is 
\[
\rho(g,r,d,\{\alpha^{l_{Y_i}}(q_{ik})\}_{ik})\geq \sum_i \rho(Y_i, \{p_{ij}\}_j,\{q_{ik}\}_k).
\]

The smoothing result \cite[Cor. 3.7]{MR846932} assures the smoothability of dimensionally proper limit series. The following facts will ease the computations. The adjusted Brill-Noether number for any $\g^r_d$ on one-pointed elliptic curves or on $n$-pointed rational curves is non-negative. For a general curve $C$ of arbitrary genus $g$, the adjusted Brill-Noether number for any $\g^r_d$ with respect to $n$ general points is non-negative. Moreover, $\rho(C,y)\geq -1$ for any $y\in C$ and any $\g^r_d$ (see \cite{MR985853}).

We will use the fact that if a curve of compact type has no limit linear series of type $\g^r_d$, then it is not in the closure of the locus $\M^r_{g,d}\subset \M_g$ of smooth curves admitting a $\g^r_d$.

\section{Test surfaces}
\label{testsurfaces}

We are going to intersect both sides of (\ref{generalclass}) with several test surfaces. This will produce linear relations in the coefficients $A$.

The surfaces will be defined for arbitrary $g\geq 6$ (also odd values). Note that while the intersections of the surfaces with the generating classes (that is the left-hand sides of the relations we get) clearly depend solely on $g$, only the right-hand sides are specific to our problem of intersecting the test surfaces with $\Mbar^1_{2k,k}$.

When the base of a family is the product of two curves $C_1 \times C_2$, we will denote by $\pi_1$ and $\pi_2$ the obvious projections. 

\vskip1pc

{\setlength{\leftmargini}{0pt} 
\begin{enumerate}
\item[(S1)] For $2\leq i \leq \lfloor g/2 \rfloor$ consider the family of curves whose fibers are obtained identifying a moving point on a general curve $C_1$ of genus $i$ with a moving point on a general curve $C_2$ of genus $g-i$. 

\begin{figure}[htbp]
\psfrag{C1}[b][b]{$C_1$}
\psfrag{C2}[b][b]{$C_2$}
\centering
  \includegraphics[scale=0.5]{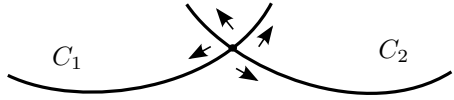}
  \caption{How the general fiber of a family in (S1) moves}
\end{figure}

The base of the family is the surface $C_1\times C_2$. In order to construct this family, consider $C_1 \times C_1 \times C_2$ and $C_1 \times C_2 \times C_2$ and identify $\Delta_{C_1}\times C_2$ with $C_1 \times \Delta_{C_2}$. Let us denote this family by $X\rightarrow C_1\times C_2$. 

One has
\begin{eqnarray*}
\delta_i &=& c_1\left( N_{(\Delta_{C_1}\times C_2)/X}\otimes N_{(C_1 \times \Delta_{C_2})/X} \right)	\\
		&=& -\pi_1^*(K_{C_1})-\pi_2^*(K_{C_2}).
\end{eqnarray*}
Such surfaces are in the interior of the boundary of $\Mbar_g$. The only nonzero classes in codimension two are the ones considered in $\S \ref{kappa&omega}$.

We claim that the intersection of these test surfaces with $\Mbar^1_{2k,k}$ has degree
\[
T_i := \mathop{\sum_{\alpha=(\alpha_0,\alpha_1)}}_{\rho(i,1,k,\alpha)=-1} n_{i,k,\alpha} \cdot n_{g-i,k,(k-1-\alpha_1,k-1-\alpha_0)}
\]
(in the sum, $\alpha$ is a Schubert index of type $1,k$).
Indeed by the remarks in \S \ref{lls}, if $\{l_{C_1},l_{C_2} \}$ is a limit linear series of type $\g^1_k$ on the fiber over some $(x,y)\in C_1\times C_2$, then the only possibility is $\rho(C_1,x)=\rho(C_2,y)=-1$. By \S \ref{movingpt}, there are exactly $T_i$ points $(x,y)$ with this property, the linear series $l_{C_1},l_{C_2}$ are uniquely determined and give an admissible cover of degree $k$. Thus to prove the claim we have to show that such points contribute with multiplicity one.

Let us first assume that $i>2$. Let $\pi \colon C' \rightarrow P$ be one of these admissible covers of degree $k$, that is, $C'$ is stably equivalent to a certain fiber $C_1\cup_{x\sim y} C_2$ of the family over $C_1\times C_2$. Let us describe more precisely the admissible covering. Note that $P$ is the union of two rational curves $P=(\PP^1)_1\cup (\PP^1)_2$. Moreover $\pi|_{C_1} \colon C_1 \rightarrow (\PP^1)_1$ is the admissible covering of degree $k-\alpha_0$ defined by $l_{C_1}(-\alpha_0 x)$, $\pi|_{C_2} \colon C_2 \rightarrow (\PP^1)_2$ is the admissible covering of degree $k-(k-1-\alpha_1)=\alpha_1+1$ defined by $l_{C_2}(-(k-1-\alpha_1) y)$, and $\pi$ has $\ell$-fold branching at $p:=x\equiv y$ with $\ell:=\alpha_1+1-\alpha_0$. Finally there are $\alpha_0$ copies of $\PP^1$ over $(\PP^1)_1$ and further $k-1-\alpha_1$ copies over $(\PP^1)_2$.

Such a cover has no automorphisms, hence the corresponding point $[\pi \colon C' \rightarrow P]$ in the Hurwitz scheme $\overline{H}_{k,b}$ is smooth, and moreover such a point is not fixed by any $\sigma \in \Sigma_b$. Let us embed $\pi \colon C' \rightarrow P$ in a one-dimensional family of admissible coverings 
\begin{eqnarray*}
\mathcal{C} \!\!\!\! & \longrightarrow & \!\!\!\! \mathcal{P} \\
&\searrow \quad \swarrow &\\
& B &
\end{eqnarray*}
where locally near the point $p$
\begin{eqnarray*}
\mathcal{C} & \mbox{is}& r\cdot s = t,\\
\mathcal{P} & \mbox{is} & u\cdot v = t^{\ell},\\
\pi 		   & \mbox{is} & u=r^{\ell}, \,\, v=s^{\ell}
\end{eqnarray*}
and $B:= {\rm Spec}\, \mathbb{C} [[t]]$.
Now $\mathcal{C}$ is a smooth surface and after contracting the extra curves $\PP^1$, we obtain a family $\mathcal{C}\rightarrow B$ in $\Mbar_g$ transverse to $\Delta_i$ at the point $[C']$. Hence $(x,y)$ appears with multiplicity one in the intersection of $\Mbar^1_{2k,k}$ with $C_1\times C_2$.

Finally if $i=2$, then one has to take into account the automorphisms of the covers. To solve this, one has to work with the universal deformation space of the corresponding curve. The argument is similar (see \cite[pg 80]{MR664324}).

For each $i$ we deduce the following relation
\[
(2i-2)(2(g-i)-2)\left[2A_{\kappa_1^2} - A_{\omega^{(i)}} - A_{\omega^{(g-i)}} \right]= T_i.
\]
Note that, if $i=g/2$, then $A_{\omega^{(i)}}$ and $A_{\omega^{(g-i)}}$ sum up.

\vskip1pc

\item[(S2)] Choose $i,j$ such that $2\leq i\leq j \leq g-3$ and $i+j\leq g-1$. Take a general two-pointed curve $(F,p,q)$ of genus $g-i-j$ and attach at $p$ a moving point on a general curve $C_1$ of genus $i$ and at $q$ a moving point on a general curve $C_2$ of genus $j$. 

\begin{figure}[htbp]
\psfrag{C1}[b][b]{$C_1$}
\psfrag{C2}[b][b]{$C_2$}
\psfrag{F}[b][b]{$F$}
\centering
  \includegraphics[scale=0.5]{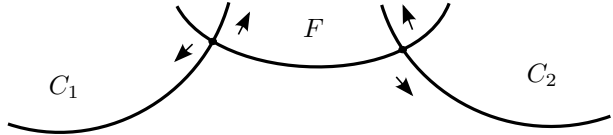}
  \caption{How the general fiber of a family in (S2) moves}
\end{figure}

The base of the family is $C_1\times C_2$. To construct the family, consider $C_1\times C_1 \times C_2$ and $C_1\times C_2\times C_2$ and identify $\Delta_{C_1}\times C_2$ and $C_1\times \Delta_{C_2}$ with the general constant sections $p\times C_1\times C_2$ and $q\times C_1\times C_2$ of $F\times C_1 \times C_2\rightarrow C_1\times C_2$. Denote this family by $X\rightarrow C_1\times C_2$.
Then
\begin{eqnarray*}
\delta_i &=& c_1\left( N_{(\Delta_{C_1}\times C_2)/X} \otimes N_{(p\times C_1\times C_2)/X}  \right) \\
	    &=& -\pi_1^*(K_{C_1})\\
\delta_j &=& c_1\left( N_{(C_1\times \Delta_{C_2})/X} \otimes N_{(q\times C_1\times C_2)/X}  \right)\\
	     &=& - \pi_2^*(K_{C_2})
\end{eqnarray*}
and
\begin{eqnarray*}
\delta_{ij} &=& c_1\left( N_{(\Delta_{C_1}\times C_2)/X} \otimes N_{(p\times C_1\times C_2)/X}  \right) \\
&& \quad \quad \quad\quad \cdot c_1\left( N_{(C_1\times \Delta_{C_2})/X} \otimes N_{(q\times C_1\times C_2)/X}  \right)\\
                &=& \pi_1^*(K_{C_1})\pi_2^*(K_{C_2}).
\end{eqnarray*}

We claim that the intersection of these test surfaces with $\Mbar^1_{2k,k}$ has degree
\begin{multline*}
D_{ij} :=  \\
\mathop{\mathop{{\mathop{\sum_{\alpha=(\alpha_0,\alpha_1)}}_{\beta=(\beta_0,\beta_1)}}}_{\rho(i,1,k,\alpha)=-1}}_{\rho(j,1,k,\beta)=-1} 
 n_{i,k,\alpha}\,  n_{j,k,\beta}  \, 
 N_{g-i-j,k,(k-1-\alpha_1,k-1-\alpha_0),(k-1-\beta_1,k-1-\beta_0)}
\end{multline*}
(in the sum, $\alpha$ and $\beta$ are Schubert indices of type $1,k$.)
Indeed by \S \ref{lls}, if $\{l_{C_1},l_F, l_{C_2} \}$ is a limit linear series of type $\g^1_k$ on the fiber over some $(x,y)\in C_1\times C_2$, then the only possibility is $\rho(C_1,x)=\rho(C_2,y)=-1$ while $\rho(F,p,q)=0$. By \S \ref{fixing2genpts} and \S \ref{movingpt}, there are
\[
\mathop{\mathop{{\mathop{\sum_{\alpha=(\alpha_0,\alpha_1)}}_{\beta=(\beta_0,\beta_1)}}}_{\rho(i,1,k,\alpha)=-1}}_{\rho(j,1,k,\beta)=-1} 
 n_{i,k,\alpha}\,  n_{j,k,\beta}
\]
points $(x,y)$ in $C_1\times C_2$ with this property, the $l_{C_1},l_{C_2}$ are uniquely determined and there are 
\[
N:= N_{g-i-j,k,(k-1-\alpha_1,k-1-\alpha_0),(k-1-\beta_1,k-1-\beta_0)}
\] 
choices for $l_F$. That is, there are $N$ points of $\overline{H}_{k,b}/\Sigma_b$ over $[C_1 \cup_{x\sim p} F \cup_{y\sim q} C_2]\in \Mbar^1_{2k,k}$ and $\Mbar^1_{2k,k}$ has $N$ branches at $[C_1 \cup_{x\sim p} F \cup_{y\sim q} C_2]$. The claim is thus equivalent to say that each branch meets $\Delta_{ij}$ transversely at $[C_1 \cup_{x\sim p} F \cup_{y\sim q} C_2]$.

The argument is similar to the previous case. Let $\pi \colon C' \rightarrow D$ be an admissible cover of degree $k$ with $C'$ stably equivalent to a certain fiber of the family over $C_1\times C_2$. The image of a general deformation of $[C'\rightarrow D]$ in $\overline{H}_{k,b}$ to the universal deformation space of $C'$ meets $\Delta_{ij}$ only at $[C']$ and locally at the two nodes, the deformation space has equation $xy=t$. Hence $[C']$ is a transverse point of intersection of $\Mbar^1_{2k,k}$ with $\Delta_{ij}$ and the surface $C_1\times C_2$ and $\Mbar^1_{2k,k}$ meet transversally.

For $i,j$ we obtain the following relation
\[
(2i-2)(2j-2) \left[ 2 A_{\kappa_1^2}+ A_{\delta_{ij}} \right] = D_{ij}.
\]

\vskip1pc

\item[(S3)] Let $(E,p,q)$ be a general two-pointed elliptic curve. Identify the point $q$ with a moving point $x$ on $E$ and identify the point $p$ with a moving point on a general curve $C$ of genus $g-2$. 

\begin{figure}[htbp]
\psfrag{C}[b][b]{$C$}
\psfrag{E}[b][b]{$E$}
\centering
  \includegraphics[scale=0.5]{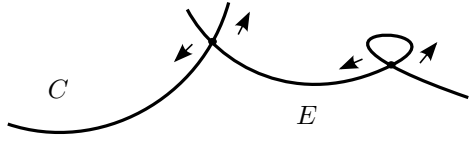}
  \caption{How the general fiber of a family in (S3) moves}
\end{figure}

The base of the family is $E\times C$. To construct the family, let us start from the blow-up $\widetilde{E\times E}$ of $E\times E$ at the points $(p,p)$ and $(q,q)$. Denote by $\sigma_p,\sigma_q,\sigma_\Delta$ the proper transforms respectively of $p\times E,q\times E,\Delta_E$. The family is the union of $\widetilde{E\times E}\times C$ and $E\times C\times C$ with $\sigma_q\times C$ identified with $\sigma_\Delta\times C$ and $\sigma_p\times C$ identified with $E\times \Delta_C$. We denote the family by $\pi\colon X\rightarrow E\times C$.

The study of the restriction of the generating classes in codimension one is similar to the case in the proof of Lemma \ref{ell}. Namely
\begin{eqnarray*}
\delta_0 &=& -\pi_1^*(2q)\\
\delta_1 &=& \pi_1^*(q)\\
\delta_{g-2} &=& -\pi_1^*(p)-\pi_2^*(K_{C}).
\end{eqnarray*}
Indeed the family is entirely contained inside $\Delta_0$: each fiber has a unique non-disconnecting node with the exception of the fibers over $p\times C$ which have two non-disconnecting nodes. Looking at the normalization of the family, fibers become smooth with the exception of the fibers over $p\times C$ which have now one non-disconnecting node, and the family is smooth at these points. It follows that $\delta_0=\pi_*(\sigma_q\times C)^2+\pi_*(\sigma_\Delta\times C)^2 +p\times C$. Only the fibers over $q\times C$ contain a node of type $\Delta_1$, and the family is smooth at these points. Finally the family is entirely inside $\Delta_{g-2}$ and $\delta_{g-2}=\pi_*(\sigma_p\times C)^2+\pi_*(E\times \Delta_C)^2$.
We note the following
\begin{eqnarray*}
\delta_{1,g-2} &=& [ \pi_1^*(q)][-\pi_2^*(K_{C})]\\
\delta_{0,g-2} &=& [-\pi_1^*(2q)][-\pi_2^*(K_{C})].
\end{eqnarray*}

Let us study the intersection of this test surface with $\Mbar^1_{2k,k}$. Let $C'\rightarrow D$ be an admissible cover of degree $k$ with $C'$ stably equivalent to a certain fiber of the family. Clearly the only possibility is to map $E$ and $C$ to two different rational components of $D$ with $q$ and $x$ in the same fiber, and have a $2$-fold ramification at $p$. From Lemma \ref{ex} there are two possibilities for the point $x\in E$, and there are $n_{g-2,k,(0,1)}$ points in $C$ where a degree $k$ covering has a $2$-fold ramification. In each case the covering is unique up to isomorphism. The combination of the two makes 
\[
2n_{g-2,k,(0,1)}
\]
admissible coverings. We claim that they count with multiplicity one.

The situation is similar to Lemma \ref{ex}. The image of a general deformation of $[C'\rightarrow D]$ in $\overline{H}_{k,b}$ to the universal deformation space of $C'$ meets $\Delta_{00}\cap \Delta_2$ only at $[C']$. Locally at the three nodes, the deformation space has equation $xy=t$. Hence $[C']$ is a transverse point of intersection of $\Mbar^1_{2k,k}$ with $\Delta_{00}\cap \Delta_2$ and counts with multiplicity one in the intersection of the surface $E\times C$ with $\Mbar^1_{2k,k}$.

We deduce the following relation
\begin{eqnarray*}
(2(g-2)-2) \Big[4A_{\kappa_1^2} -A_{\omega^{(2)}}-A_{\omega^{(g-2)}}
- A_{\delta_{1,g-2}} + 2 A_{\delta_{0,g-2}}
 \Big]= 2n_{g-2,k,(0,1)}.
\end{eqnarray*}

\vskip1pc

\item[(S4)] For $2\leq i \leq g-3$, let $(F,r,s)$ be a general two-pointed curve of genus $g-i-2$. Let $(E,p,q)$ be a general two-pointed elliptic curve and as above identify the point $q$ with a moving point $x$ on $E$. Finally identify the point $p\in E$ with $r\in F$ and identify the point $s\in F$ with a moving point on a general curve $C$ of genus $i$. 

\begin{figure}[htbp]
\psfrag{C}[b][b]{$C$}
\psfrag{E}[b][b]{$E$}
\psfrag{F}[b][b]{$F$}
\centering
  \includegraphics[scale=0.5]{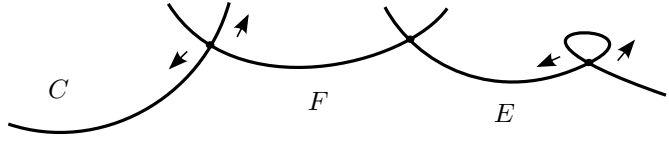}
  \caption{How the general fiber of a family in (S4) moves}
\end{figure}

The base of the family is $E\times C$. Let $\widetilde{E\times E}, \sigma_p,\sigma_q,\sigma_\Delta$ be as above. Then the family is the union of $\widetilde{E\times E}\times C, E\times C\times C$ and $F\times E\times C$ with the following identifications. First $\sigma_q\times C$ is identified with $\sigma_\Delta\times C$. Finally $\sigma_p\times E$ is identified with $r\times E\times C\subset F\times E\times C$, and $s\times E\times C\subset F\times E\times C$ with $E\times \Delta_C$.

The restriction of the generating classes in codimension one is 
\begin{eqnarray*}
\delta_0 &=& -\pi_1^*(2q)\\
\delta_1 &=& \pi_1^*(q)\\
\delta_2 &=& -\pi_1^*(p)\\
\delta_i &=& -\pi_2^*(K_{C})
\end{eqnarray*}
and one has the following restrictions
\begin{eqnarray*}
\delta_{1,i} &=& [\pi_1^*(q)][-\pi_2^*(K_{C})]\\
\delta_{0,i} &=& [-\pi_1^*(2q)][-\pi_2^*(K_{C})]\\
\delta_{2,i} &=& [-\pi_1^*(p)][-\pi_2^*(K_{C})].
\end{eqnarray*}

Suppose $C'\rightarrow D$ is an admissible covering of degree $k$ with $C'$ stably equivalent to a certain fiber of this family. The only possibility is to map $E,F,C$ to three different rational components of $D$, with a $2$-fold ramification at $r$ and ramification prescribed by $\alpha=(\alpha_0,\alpha_1)$ at $s$, such that $\rho(i,1,k,\alpha)=-1$. The condition on $\alpha$ is equivalent to 
\[
\rho(g-i-2,1,k,(0,1),(k-1-\alpha_1,k-1-\alpha_0))=0.
\]
Moreover, $q$ and $x$ have to be in the same fiber of such a covering.
There are 
\[
\mathop{\sum_{\alpha=(\alpha_0,\alpha_1)}}_{\rho(i,1,k,\alpha)=-1} 2 n_{i,k,\alpha}
\]
fibers which admit an admissible covering with such properties (in the sum, $\alpha$ is a Schubert index of type $1,k$). While the restriction of the covering to $E$ and $C$ is uniquely determined up to isomorphism, there are 
\[
N:= N_{g-i-2,k,(0,1),(k-1-\alpha_1,k-1-\alpha_0)}
\]
choices for the restriction to $F$ up to isomorphism. As in (S2), this is equivalent to say that $\Mbar^1_{2k,k}$ has $N$ branches at $[C']$. Moreover, each branch meets the boundary transversally at $[C']$ (similarly to (S3)), hence $[C']$ counts with multiplicity one in the intersection of $E\times C$ with $\Mbar^1_{2k,k}$.

Finally, for each $i$ we deduce the following relation
\begin{multline*}
(2i-2) \Big[4A_{\kappa_1^2} - A_{\delta_{1,i}} + 2 A_{\delta_{0,i}} +A_{\delta_{2,i}} \Big]\\
 = \mathop{\sum_{\alpha=(\alpha_0,\alpha_1)}}_{\rho(i,1,k,\alpha)=-1} 2N_{g-i-2,k,(0,1),(k-1-\alpha_1,k-1-\alpha_0)}\cdot n_{i,k,\alpha}.
\end{multline*}

\vskip1pc
\item[(S5)] Identify a base point of a generic pencil of plane cubic curves with a moving point on a general curve $C$ of genus $g-1$. 

\begin{figure}[htbp]
\psfrag{C}[b][b]{$C$}
\centering
  \includegraphics[scale=0.5]{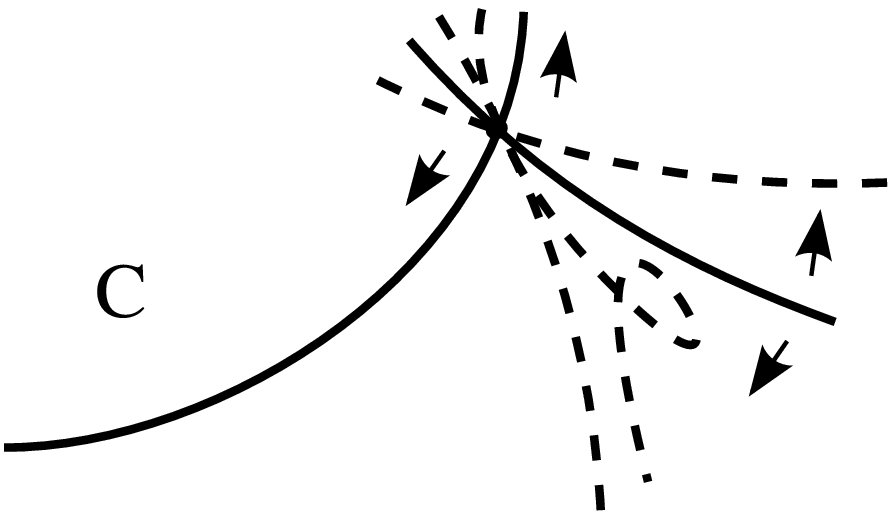}
  \caption{How the general fiber of a family in (S5) moves}
\end{figure}

The base of the family is $\PP^1 \times C$. Let us construct this family. We start from an elliptic pencil $Y\rightarrow \PP^1$ of degree $12$ with zero section $\sigma$. To construct $Y$, blow up $\PP^2$ in the nine points of intersection of two general cubics. Then consider $Y\times C$ and $\PP^1 \times C \times C$ and identify $\sigma\times C$ with $\PP^1 \times \Delta_C$. Let $x$ be the class of a point in $\PP^1$. 
Then
\begin{eqnarray*}
\lambda &=& \pi_1^*(x)\\
\delta_0 &=& 12\lambda\\
\delta_1 &=& -\pi_1^*(x)-\pi_2^*(K_C).
\end{eqnarray*}
Note that
\[
\delta_{0,g-1} = [12\pi_1^*(x)] [-\pi_2^*(K_C)].
\]

This surface is disjoint from $\Mbar^1_{2k,k}$. Indeed $C$ has no linear series with adjusted Brill-Noether number less than $-1$ at some point, and an elliptic curve or a rational nodal curve has no (generalized) linear series with adjusted Brill-Noether number less than $0$ at some point. Adding, we see that no fiber of the family has a linear series with Brill-Noether number less than $-1$,
hence
\[
(2(g-1)-2)\left[ 2A_{\kappa_1^2}-12 A_{\delta_{0,g-1}}+ 2A_{\delta_1^2} -A_{\lambda\delta_1} \right]=0.
\]

\vskip1pc
\item[(S6)]  For $3\leq i \leq g-3$ take a general curve $F$ of genus $i-1$ and attach at a general point $p$ an elliptic tail varying in a pencil of degree $12$ and at another general point a moving point on a general curve $C$ of genus $g-i$. 

\begin{figure}[htbp]
\psfrag{C}[b][b]{$C$}
\psfrag{F}[b][b]{$F$}
\centering
  \includegraphics[scale=0.5]{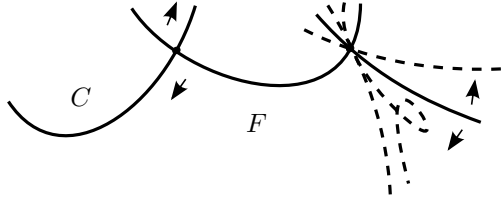}
  \caption{How the general fiber of a family in (S6) moves}
\end{figure}

The base of the family is $\PP^1 \times C$. In order to construct the family, start from $Y \times C$ and $\PP^1\times C\times C$ and identify $\sigma \times C$ and $\PP^1 \times \Delta_C$ with two general constant sections of $F\times \PP^1 \times C \rightarrow \PP^1\times C$. Here $Y,\sigma$ are as above.
Then
\begin{eqnarray*}
\lambda &=& \pi_1^*(x)\\
\delta_0 &=& 12 \lambda \\
\delta_1 &=& -\pi_1^*(x)\\
\delta_{g-i} &=& -\pi_2^*(K_{C}).
\end{eqnarray*}
Note that
\begin{eqnarray*}
\delta_{1,g-i} &=& [-\pi_1^*(x)][-\pi_2^*(K_{C})]\\
\delta_{0,g-i} &=& [12 \pi_1^*(x)][-\pi_2^*(K_{C})] .
\end{eqnarray*}

Again $C$ has no linear series with adjusted Brill-Noether number less than $-1$ at some point, an elliptic curve or a rational nodal curve has no (generalized) linear series with adjusted Brill-Noether number less than $0$ at some point and $F$ has no linear series with adjusted Brill-Noether number less than $0$ at some general points. Adding, we see that no fiber of the family has a linear series with Brill-Noether number less than $-1$,
hence
\begin{eqnarray*}
 (2(g-i)-2)\left[ 2 A_{\kappa_1^2} - A_{\lambda^{(i)}} + A_{\delta_{1,g-i}} - 12 A_{\delta_{0,g-i}} \right] = 0. 
\end{eqnarray*}

\noindent In case $i=g-2$ we have
\[
2\left[2A_{\kappa_1^2} - A_{\lambda \delta_2}+ A_{\delta_{1,2}} - 12 A_{\delta_{0,2}} \right]=0.
\]

\vskip1pc
\item[(S7)] Let $(E_1,p_1,q_1)$ and $(E_2,p_2,q_2)$ be two general pointed elliptic curves. Identify the point $q_i$ with a moving point $x_i$ in $E_i$, for $i=1,2$. Finally identify $p_1$ and $p_2$ with two general points $r_1,r_2$ on a general curve $F$ of genus $g-4$.

\begin{figure}[htbp]
\psfrag{F}[b][b]{$F$}
\psfrag{E1}[b][b]{$E_1$}
\psfrag{E2}[b][b]{$E_2$}
\centering
  \includegraphics[scale=0.5]{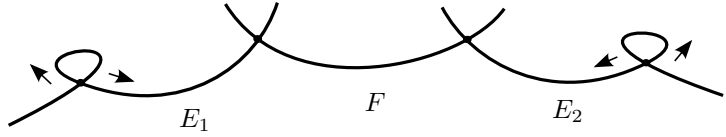}
  \caption{How the general fiber of a family in (S7) moves}
\end{figure}

The base of the family is $E_1\times E_2$. For $i=1,2$, let $\widetilde{E_i\times E_i}$ be the blow-up of $E_i\times E_i$ at $(p_i,p_i)$ and $(q_i,q_i)$. Denote by $\sigma_{p_i},\sigma_{q_i},\sigma_{\Delta_{E_i}}$ the proper transforms respectively of $p_i\times E_i,q_i\times E_i, \Delta_{E_i}$. The family is the union of $\widetilde{E_1\times E_1}\times E_2$, $E_1\times \widetilde{E_2\times E_2}$ and $F\times E_1\times E_2$ with the following identifications. First, $\sigma_{q_1}\times E_2$ and $E_1\times \sigma_{q_2}$ are identified respectively with $\sigma_{\Delta_{E_1}}\times E_2$ and $E_1\times \sigma_{\Delta_{E_2}}$. Then $\sigma_{p_1}\times E_2$ and $E_1\times \sigma_{p_2}$ are identified respectively with $r_1\times E_1\times E_2$ and $r_2\times E_1\times E_2$.
We deduce
\begin{eqnarray*}
\delta_0 &=& -\pi_1^*(2q_1)-\pi_2^*(2q_2)\\
\delta_1 &=& \pi_1^*(q_1)+\pi_2^*(q_2)\\
\delta_2 &=& -\pi_1^*(p_1)-\pi_2^*(p_2)
\end{eqnarray*}
and we note that
\begin{eqnarray*}
\delta_{2,2} &=& \pi_1^*(p_1) \pi_2^*(p_2) \\
\delta_{1,2} &=& -\pi_1^*(q_1)\pi_2^*(p_2)-\pi_2^*(q_2)\pi_1^*(p_1) \\
\delta_{1,1} &=& \pi_1^*(q_1) \pi_2^*(q_2)\\
\delta_{00} &=& \pi_1^*(2q_1)\pi_2^*(2q_2) \\
\delta_{02} &=& \pi_1^*(2q_1)\pi_2^*(p_2)+ \pi_2^*(2q_2)\pi_1^*(p_1)\\
\delta_{01} &=& -\pi_1^*(q_1)\pi_2^*(2q_2)-\pi_2^*(q_2)\pi_1^*(2q_1) .
\end{eqnarray*}

If a fiber of this family admits an admissible cover of degree $k$, then $r_1$ and $r_2$ have to be $2$-fold ramification points, and $q_i$ and $x_i$ have to be in the same fiber, for $i=1,2$. From Lemma \ref{ex} there are only $4$ fibers with this property, namely the fibers over $(p_1,p_2)$, $(p_1,\overline{q}_2)$, $(\overline{q}_1,p_2)$ and $(\overline{q}_1,\overline{q}_2)$, where $\overline{q}_i$ is such that $2p_i\equiv q_i+\overline{q}_i$ for $i=1,2$.

In these cases, the restriction of the covers to $E_1,E_2$ is uniquely determined up to isomorphism, while there are $N_{g-4,k,(0,1),(0,1)}$ choices for the restriction to $F$ up to isomorphism. As for (S3), such covers contribute with multiplicity one, hence we have the following relation
\begin{multline*}
8A_{\kappa_1^2} + A_{\delta_{2,2}} -2A_{\delta_{1,2}} + A_{\delta_{1,1}} + 2A_{\delta_1^2}+ 8A_{\delta_0^2} + 4 A_{\delta_{00}} + 4 A_{\delta_{02}}-4A_{\delta_{01}} \\
= 4 N_{g-4,k,(0,1),(0,1)}.
\end{multline*}

\vskip1pc
\item[(S8)] Consider a general curve $F$ of genus $g-2$ and attach at two general points elliptic tails varying in pencils of degree $12$. 

\begin{figure}[htbp]
\psfrag{F}[b][b]{$F$}
\centering
  \includegraphics[scale=0.5]{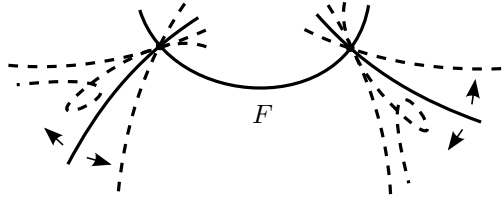}
  \caption{How the general fiber of a family in (S8) moves}
\end{figure}

The base of the family is $\PP^1 \times \PP^1$. Let us construct the family. Let $Y\rightarrow \PP^1$ and $Y'\rightarrow \PP^1$ be two elliptic pencils of degree $12$, and let $\sigma$ and $\sigma'$ be the respective zero sections. Consider $Y\times \PP^1$ and $\PP^1\times Y'$ and identify $\sigma\times \PP^1$ and $\PP^1 \times \sigma'$ with two general constant sections of $F\times \PP^1 \times \PP^1\rightarrow \PP^1\times \PP^1$.
If $x$ is the class of a point in $\PP^1$, then
\begin{eqnarray*}
\lambda &=& \pi_1^*(x)+\pi_2^*(x)\\
\delta_0 &=& 12 \lambda\\
\delta_1 &=& -\lambda.
\end{eqnarray*}
Note that
\begin{eqnarray*}
\delta_{00} &=& [12 \pi_1^*(x)][12\pi_2^*(x)] \\
\delta_{1,1} &=& [-\pi_1^*(x)][-\pi_2^*(x)]\\
\delta_{01} &=& [12 \pi_1^*(x)][-\pi_2^*(x)] +[-\pi_1^*(x)] [12 \pi_2^*(x)].
\end{eqnarray*}

Studying the possibilities for the adjusted Brill-Noether numbers of the aspects of limit linear series on some fiber of this family, we see that this surface is disjoint from $\Mbar^1_{2k,k}$, hence
\[
2A_{\kappa_1^2} +288A_{\delta_0^2}+24A_{\lambda \delta_0}+2A_{\delta_1^2}-2A_{\lambda \delta_1}+ 144A_{\delta_{00}}+ A_{\delta_{1,1}} -24 A_{\delta_{01}}=0.
\]

\vskip1pc
\item[(S9)] For $2\leq j \leq g-3$ let $R$ be a smooth rational curve, attach at the point $\infty\in R$ a general curve $F$ of genus $g-j-2$, attach at the points $0,1 \in R$ two elliptic tails $E_1,E_2$ and identify a moving point in $R$ with a moving point on a general curve $C$ of genus $j$. 

\begin{figure}[htbp]
\psfrag{F}[b][b]{$F$}
\psfrag{R}[b][b]{$R$}
\psfrag{C}[b][b]{$C$}
\psfrag{E1}[b][b]{$E_1$}
\psfrag{E2}[b][b]{$E_2$}
\centering
  \includegraphics[scale=0.5]{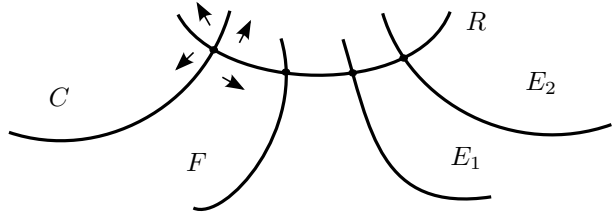}
  \caption{How the general fiber of a family in (S9) moves}
\end{figure}

The base of the family is $R \times C$. Let us start from a family $P\rightarrow R$ of four-pointed rational curves. Construct $P$ by blowing up $\PP^1\times \PP^1$ at $(0,0),(1,1)$ and $(\infty,\infty)$, and consider the sections $\sigma_0, \sigma_1,\sigma_{\infty}$ and $\sigma_{\Delta}$ corresponding to the proper transforms of $0\times \PP^1, 1\times \PP^1, \infty\times \PP^1$ and $\Delta_{\PP^1}$.

To construct the family over $R \times C$, consider $P\times C$ and $R\times C\times C$. Identify $\sigma_\Delta\times C$ with $R\times \Delta_C$. Finally identify $\sigma_0\times C,\sigma_1\times C$ and $\sigma_\infty\times C$ respectively with general constant sections of the families $E_1\times R\times C, E_2\times R\times C$ and $F\times R\times C$.
Then 
\begin{eqnarray*}
\delta_1 &=& -\pi_1^*(0+1)\\
\delta_2 &=& \pi_1^*(\infty)\\
\delta_j &=& -\pi_1^*(K_{\PP^1}+0+1+\infty)-\pi_2^*(K_{C})\\
\delta_{g-j-2} &=& -\pi_1^*(\infty)\\
\delta_{g-j-1} &=& \pi_1^*(0+1).
\end{eqnarray*}
If for some value of $j$ some of the above classes coincide (for instance, if $j=g-3$ then $\delta_1\equiv \delta_{g-j-2}$), then one has to sum up the contributions.
Note that
\begin{eqnarray*}
\delta_{1j} &=& [-\pi_1^*(0+1)][-\pi_2^*(K_{C})]\\
\delta_{j,g-j-2} &=& [-\pi_1^*(\infty)][-\pi_2^*(K_{C})] \\
\delta_{2,j} &=& [\pi_1^*(\infty)][-\pi_2^*(K_{C})]\\
\delta_{j,g-j-1} &=& [ \pi_1^*(0+1)][-\pi_2^*(K_{C})].
\end{eqnarray*}
As for (S8), this surface is disjoint from $\Mbar^1_{2k,k}$, hence
\begin{multline*}
(2j-2) \Big[ 2 A_{\kappa_1^2}+ 2  A_{\delta_{1j}}+A_{\delta_{j,g-j-2}}-A_{\delta_{2,j}}
-2A_{\delta_{j,g-j-1}}\\
{}- A_{\omega^{(j)}}-A_{\omega^{(g-j)}} \Big] = 0.
\end{multline*}
Again, let us remark that for some value of $j$, some terms add up.

\vskip1pc
\item[(S10)] Let $(R_1, 0, 1, \infty)$ and $(R_2, 0, 1, \infty)$ be two three-pointed smooth rational curves, identify a moving point on $R_1$ with a moving point on $R_2$, attach a general pointed curve $F$ of genus $g-5$ to $\infty\in R_2$ and attach elliptic tails to all the other marked points. 

\begin{figure}[htbp]
\psfrag{F}[b][b]{$F$}
\psfrag{R2}[b][b]{$R_2$}
\psfrag{R1}[b][b]{$R_1$}
\centering
  \includegraphics[scale=0.5]{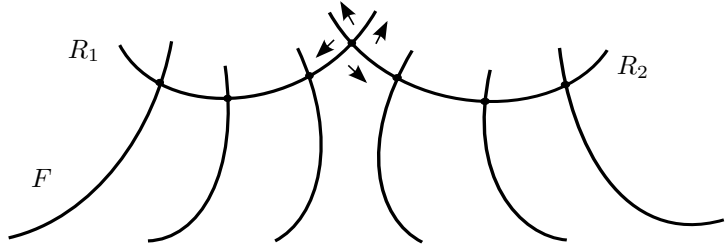}
  \caption{How the general fiber of a family in (S10) moves}
\end{figure}

The base of the family is $R_1\times R_2$. First construct two families of four-pointed rational curves $P_1\rightarrow R_1$ and $P_2\rightarrow R_2$ respectively with sections $\sigma_0, \sigma_1,\sigma_{\infty},\sigma_{\Delta}$ and $\tau_0, \tau_1,\tau_{\infty},\tau_{\Delta}$ as for the previous surface. Consider $P_1\times R_2$ and $R_1\times P_2$. Identify $\sigma_\Delta\times R_2$ with $R_1\times \tau_{\Delta}$. Finally identify $R_1\times \tau_{\infty}$ with a general constant section of $F\times R_1 \times R_2$ and identify $\sigma_0\times R_2,\sigma_1\times R_2,\sigma_\infty\times R_2,R_1\times \tau_0,R_1\times \tau_1$ with the respective zero sections of five constant elliptic fibrations over $R_1\times R_2$.

This surface is disjoint from $\Mbar^1_{2k,k}$.
For $g>8$
\begin{eqnarray*}
\delta_1 &=& -\pi_1^*(0+1+\infty)-\pi_2^*(0+1)\\
\delta_2 &=& \pi_1^*(0+1+\infty)+\pi_2^*(\infty)\\
\delta_3 &=& -\pi_1^*(K_{R_1}+0+1+\infty)-\pi_2^*(K_{R_2}+0+1+\infty)\\
\delta_{g-5} &=& -\pi_2^*(\infty)\\
\delta_{g-4} &=& \pi_2^*(0+1)
\end{eqnarray*}
and note the restriction of the following classes
\begin{eqnarray*}
\delta_{1,1} &=& [-\pi_1^*(0+1+\infty)][-\pi_2^*(0+1)] \\
\delta_{1,g-5} &=&  [-\pi_1^*(0+1+\infty)][-\pi_2^*(\infty)]\\
\delta_{1,3} &=&	[ -\pi_1^*(K_{R_1}+0+1+\infty)][-\pi_2^*(0+1)]	\\
\delta_{3,g-5} &=& 	[ -\pi_1^*(K_{R_1}+0+1+\infty)][ -\pi_2^*(\infty)]\\
\delta_{1,g-3} &=& [-\pi_1^*(0+1+\infty)][-\pi_2^*(K_{R_2}+0+1+\infty)]	\\
\delta_{2,g-3} &=& [\pi_1^*(0+1+\infty)][-\pi_2^*(K_{R_2}+0+1+\infty)] \\
\delta_{2,g-5} &=& [\pi_1^*(0+1+\infty)][ -\pi_2^*(\infty)]\\
\delta_{1,2} &=& [\pi_1^*(0+1+\infty)][-\pi_2^*(0+1)]+[-\pi_1^*(0+1+\infty)][\pi_2^*(\infty)]\\
\delta_{1,g-4} &=& [-\pi_1^*(0+1+\infty)][ \pi_2^*(0+1)]\\
\delta_{3,g-4} &=& [-\pi_1^*(K_{R_1}+0+1+\infty)][ \pi_2^*(0+1)]\\
\delta_{2,3} &=& [-\pi_1^*(K_{R_1}+0+1+\infty)][\pi_2^*(\infty)]\\
\delta_{2,g-4} &=& [ \pi_1^*(0+1+\infty)][\pi_2^*(0+1)]\\
\delta_{2,2} &=& [\pi_1^*(0+1+\infty)][\pi_2^*(\infty)].
\end{eqnarray*}
It follows that
\begin{multline*}
2A_{\kappa_1^2}+12 A_{\delta_1^2} + 6 A_{\delta_{1,1}} +3A_{\delta_{1,g-5}}+2A_{\delta_{1,3}}+A_{\delta_{3,g-5}}+3A_{\delta_{1,g-3}}\\
-A_{\omega^{(3)}}-A_{\omega^{(g-3)}} - 3(A_{\delta_{2,g-3}} +A_{\delta_{2,g-5}}+2A_{\delta_{1,2}})\\
 -2(3A_{\delta_{1,g-4}}+A_{\delta_{3,g-4}}) - (3A_{\delta_{1,2}}+A_{\delta_{2,3}})
  +6A_{\delta_{2,g-4}}+ 3A_{\delta_{2,2}} =0.
\end{multline*}
For $g=6$ the coefficient of $A_{\delta_1^2}$ is $18$. When $g\in \{6,7,8\}$, note that some terms add up.

\vskip1pc
\item[(S11)] Consider a general curve $F$ of genus $g-4$, attach at a general point an elliptic tail varying in a pencil of degree $12$ and identify a second general point with a moving point on a rational three-pointed curve $(R,0,1,\infty)$. Attach elliptic tails at the marked point on the rational curve. 

\begin{figure}[htbp]
\psfrag{F}[b][b]{$F$}
\psfrag{R}[b][b]{$R$}
\centering
  \includegraphics[scale=0.5]{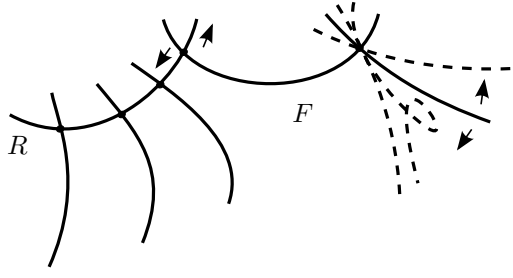}
  \caption{How the general fiber of a family in (S11) moves}
\end{figure}

The base of the family is $\PP^1\times R$. Consider the elliptic fibration $Y$ over $\PP^1$ with zero section $\sigma$ as in (S5), and the family $P$ over $R$ with sections $\sigma_0,\sigma_1,\sigma_\infty,\sigma_\Delta$ as in (S9). Identify $\sigma\times R \subset Y\times R$ and $\PP^1\times \sigma_\Delta\subset \PP^1 \times P$ with two general constant sections of $F\times \PP^1 \times R$. Finally identify $\PP^1\times \sigma_0,\PP^1\times \sigma_1,\PP^1\times \sigma_\infty\subset \PP^1 \times P$ with the respective zero sections of three constant elliptic fibrations over $\PP^1 \times R$.
Then
\begin{eqnarray*}
\lambda &=& \pi_1^*(x)\\
\delta_{0} &=& 12 \lambda \\
\delta_1 &=& -\pi_1^*(x)-\pi_2^*(0+1+\infty)\\
\delta_2 &=& \pi_2^*(0+1+\infty)\\
\delta_3 &=& -\pi_2^*(K_R +0+1+\infty).
\end{eqnarray*}
Note the restriction of the following classes
\begin{eqnarray*}
\delta_{1,1} &=& [-\pi_1^*(x)][-\pi_2^*(0+1+\infty)]\\
\delta_{1,3} &=& [-\pi_1^*(x)][ -\pi_2^*(K_R +0+1+\infty)]\\
\delta_{01} &=& [12\pi_1^*(x)][-\pi_2^*(0+1+\infty)]\\
\delta_{03} &=& [12\pi_1^*(x)][-\pi_2^*(K_R +0+1+\infty)]\\
\delta_{02} &=&  [12\pi_1^*(x)][\pi_2^*(0+1+\infty)]\\
\delta_{1,2} &=& [ -\pi_1^*(x)][\pi_2^*(0+1+\infty)].
\end{eqnarray*}
This surface is disjoint from $\Mbar^1_{2k,k}$, hence
\begin{multline*}
2 A_{\kappa_1^2} -A_{\lambda^{(g-3)}}+6A_{\delta_1^2}+3A_{\delta_{1,1}}-3 A_{\lambda \delta_1}+A_{\delta_{1,3}}-36A_{\delta_{01}}-12A_{\delta_{03}} \\
+3\left[A_{\lambda\delta_2}+12A_{\delta_{02}}-A_{\delta_{1,2}}\right] =0.
\end{multline*}

\vskip1pc
\item[(S12)] Let $R$ be a rational curve, attach at the points $0$ and $1$ two fixed elliptic tails, attach at the point $\infty$ an elliptic tail moving in a pencil of degree $12$ and identify a moving point in $R$ with a general point on a general curve $F$ of genus $g-3$. 

\begin{figure}[htbp]
\psfrag{F}[b][b]{$F$}
\psfrag{R}[b][b]{$R$}
\centering
  \includegraphics[scale=0.5]{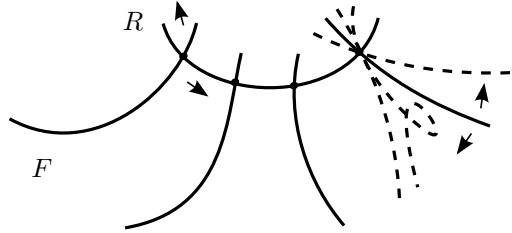}
  \caption{How the general fiber of a family in (S12) moves}
\end{figure}

The base of the family is $\PP^1 \times R$. Let $Y,\sigma$ and $P,\sigma_0,\sigma_1,\sigma_\infty,\sigma_\Delta$ be as above. Identify $\sigma\times R\subset Y\times R$ with $\PP^1\times \sigma_\infty\subset \PP^1 \times P$, and $\PP^1\times \sigma_\Delta\subset \PP^1\times P$ with a general constant section of $F\times\PP^1\times R$. Finally identify $\PP^1\times \sigma_0,\PP^1\times \sigma_1$ with the zero sections of two constant elliptic fibrations over $\PP^1\times R$.
Then
\begin{eqnarray*}
\lambda &=& \pi_1^*(x)\\
\delta_0 &=& 12 \lambda\\
\delta_1 &=& -\pi_1^*(x)-\pi_2^*(\infty+0+1)\\
\delta_2 &=& \pi_2^*(\infty+0+1)\\
\delta_3 &=& -\pi_2^*(K_{\PP^1}+0+1+\infty).
\end{eqnarray*}
Let us note the following restrictions
\begin{eqnarray*}
\delta_{01} &=& [12\pi_1^*(x)][-\pi_2^*(0+1)]\\
\delta_{0,g-3} &=& [12\pi_1^*(x)][-\pi_2^*(K_{\PP^1}+0+1+\infty)]\\
\delta_{0,g-1} &=&  [12\pi_1^*(x)][-\pi_2^*(\infty)]\\
\delta_{1,1} &=&  [-\pi_1^*(x)][-\pi_2^*(0+1)]\\
\delta_{1,g-3} &=& [-\pi_1^*(x)][-\pi_2^*(K_{\PP^1}+0+1+\infty)]\\
\delta_{0,g-2} &=& [12\pi_1^*(x)][ \pi_2^*(0+1)]\\
\delta_{1,g-2} &=& [-\pi_1^*(x)][ \pi_2^*(0+1)]\\
\delta_{02} &=& [12\pi_1^*(x)][\pi_2^*(\infty)]\\
\delta_{1,2} &=& [-\pi_1^*(x)][ \pi_2^*(\infty)].
\end{eqnarray*}
This surface is disjoint from $\Mbar^1_{2k,k}$, hence
\begin{multline*}
2A_{\kappa_1^2}-3A_{\lambda \delta_1}-24A_{\delta_{01}}-12A_{\delta_{0,g-3}}-12A_{\delta_{0,g-1}}+6A_{\delta_1^2}+2A_{\delta_{1,1}}+A_{\delta_{1,g-3}}\\
-A_{\lambda^{(3)}} +2(A_{\lambda\delta_2}+12A_{\delta_{0,g-2}}-A_{\delta_{1,g-2}})+(A_{\lambda\delta_2}+12A_{\delta_{02}}-A_{\delta_{1,2}})=0.
\end{multline*}

\vskip1pc
\item[(S13)]  Let $(C,p,q)$ be a general two-pointed curve of genus $g-3$ and identify the point $q$ with a moving point $x$ on $C$. Let $(E,r,s)$ be a general two-pointed elliptic curve and identify the point $s$ with a moving point $y$ on $E$. Finally identify the points $p$ and $r$. 

\begin{figure}[htbp]
\psfrag{C}[b][b]{$C$}
\psfrag{E}[b][b]{$E$}
\centering
  \includegraphics[scale=0.5]{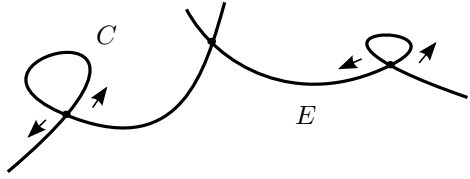}
  \caption{How the general fiber of a family in (S13) moves}
\end{figure}

The base of the family is $C\times E$. Let $\widetilde{C\times C}$ (respectively $\widetilde{E\times E}$) be the blow-up of $C\times C$ at $(p,p)$ and $(q,q)$ (respectively of $E\times E$ at $(r,r)$ and $(s,s)$). Let $\tau_p,\tau_q,\tau_\Delta$ (respectively $\sigma_r,\sigma_s,\sigma_\Delta$) be the proper transform of $p\times C,q\times C,\Delta_C$ (respectively $r\times E,s\times E,\Delta_E$) and identify $\tau_q$ with $\tau_\Delta$ (respectively $\sigma_s$ with $\sigma_\Delta$). Finally identify $\tau_p\times E$ with $C\times \sigma_r$.
Then from the proof of Lemma \ref{ell}, we have 
\begin{eqnarray*}
\delta_0 &=& -\pi_1^*(K_C+2q)-\pi_2^*(2s)\\
\delta_1 &=& \pi_1^*(q) + \pi_2^*(s)\\
\delta_2 &=& -\pi_1^*(p)-\pi_2^*(r)
\end{eqnarray*}
and note that
\begin{eqnarray*}
\delta_{00} &=& [-\pi_1^*(K_C+2q)][-\pi_2^*(2s)] \\
\delta_{02} &=& [-\pi_1^*(K_C+2q)][-\pi_2^*(r)] \\
\delta_{0,g-2} &=& [-\pi_2^*(2s)][-\pi_1^*(p)] \\
\delta_{01} &=& [ -\pi_1^*(K_C+2q)][ \pi_2^*(s)]+[-\pi_2^*(2s)][ \pi_1^*(q)] \\
\delta_{1,g-2} &=& [-\pi_1^*(p)][\pi_2^*(s)] \\
\delta_{1,2} &=& [ \pi_1^*(q)][-\pi_2^*(r)] \\
\delta_{1,1} &=& [ \pi_1^*(q)][\pi_2^*(s)] .
\end{eqnarray*}

If a fiber of this family admits an admissible covering of degree $k$, then such a covering has a $2$-fold ramification at the point $p\sim r$, $q$ is in the same fiber as $x$, and $s$ is in the same fiber as $y$. By Lemma \ref{ex} and Lemma \ref{ell} there are $2$ points in $E$ and $\ell_{g-2,k}$ points in $C$ with such a property, and the cover is unique up to isomorphism. Reasoning as in (S3), one shows that each cover contributes with multiplicity one.
It follows that
\begin{multline*}
2(g-3)\left[4A_{\kappa_1^2} +2A_{\delta_{00}}+4A_{\delta_0^2}+A_{\delta_{02}} \right] +2A_{\delta_{0,g-2}}-A_{\omega^{(2)}}-A_{\omega^{(g-2)}}\\
-\left[2(g-3)A_{\delta_{01}}+A_{\delta_{1,g-2}}  \right] -\left[2A_{\delta_{01}}+A_{\delta_{1,2}}  \right] +\left[ A_{\delta_{1,1}}+2A_{\delta_1^2} \right] = 2  \cdot \ell_{g-2,k}.
\end{multline*}

\vskip1pc
\item[(S14)] Let $(C,p,q)$ be a general two-pointed curve of genus $g-2$, attach at $p$ an elliptic tail moving in a pencil of degree $12$ and identify $q$ with a moving point on $C$. 

\begin{figure}[htbp]
\psfrag{C}[b][b]{$C$}
\centering
  \includegraphics[scale=0.5]{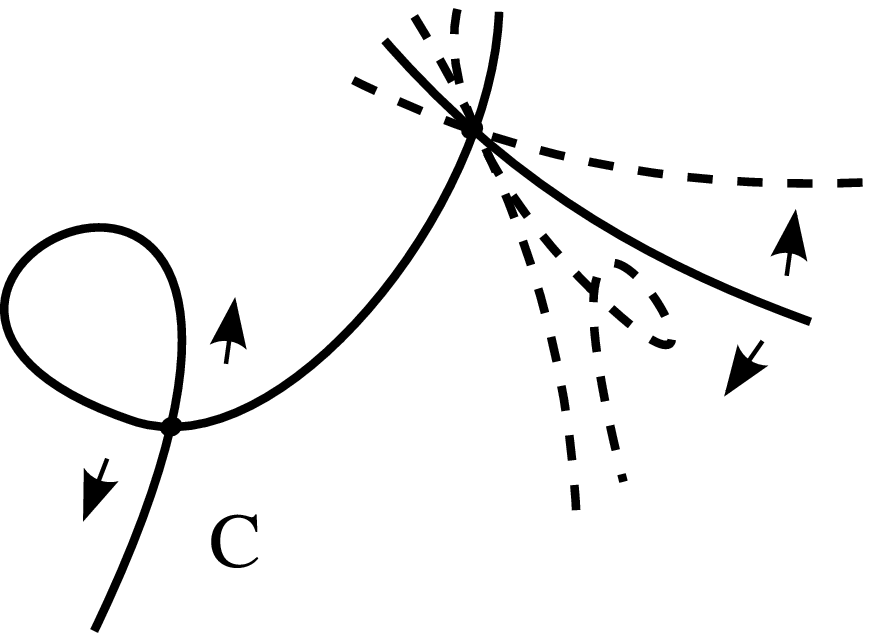}
  \caption{How the general fiber of a family in (S14) moves}
\end{figure}

The base of this family is $C\times \PP^1$. Let $\widetilde{C\times C}$ be the blow-up of $C\times C$ at the points $(p,p)$ and $(q,q)$. Let $\tau_p,\tau_q,\tau_\Delta$ be the proper transform of $p\times C,q\times C,\Delta$ and identify $\tau_q$ with $\tau_\Delta$. Then consider $Y,\sigma$ as in (S5) and identify $C\times\sigma$ with $\tau_p\times \PP^1$.
Then
\begin{eqnarray*}
\lambda &=& \pi_2^*(x)\\
\delta_0 &=& 12\lambda -\pi_1^*(K_C+2q)\\
\delta_1 &=& \pi_1^*(q)-\pi_1^*(p)-\lambda.
\end{eqnarray*}
Note that
\begin{eqnarray*}
\delta_{00} &=& [12\pi_2^*(x)][ -\pi_1^*(K_C+2q)]\\
\delta_{01} &=& [ \pi_1^*(q)][12 \pi_2^*(x)]+[-\pi_1^*(K_C+2q)][-\pi_2^*(x)]\\
\delta_{0,g-1} &=& [-\pi_1^*(p)][12\pi_2^*(x)]\\
\delta_{1,1} &=& [ \pi_1^*(q)][- \pi_2^*(x)].
\end{eqnarray*}
This surface is disjoint from $\Mbar^1_{2k,k}$, hence
\begin{multline*}
(2g-4)\left[2A_{\kappa_1^2}-A_{\lambda \delta_0} -24A_{\delta_0^2}-12A_{\delta_{00}} +A_{\delta_{01}}\right]\\
-12A_{\delta_{0,g-1}} + (12A_{\delta_{01}}-A_{\delta_{1,1}})
=0.
\end{multline*}

\vskip1pc

\item[(S15)] Let $C$ be a general curve of genus $g-1$ and consider the surface $C \times C$ with fiber $C/(p\sim q)$ over $(p,q)$. 

\begin{figure}[htbp]
\psfrag{C}[b][b]{$C$}
\centering
  \includegraphics[scale=0.5]{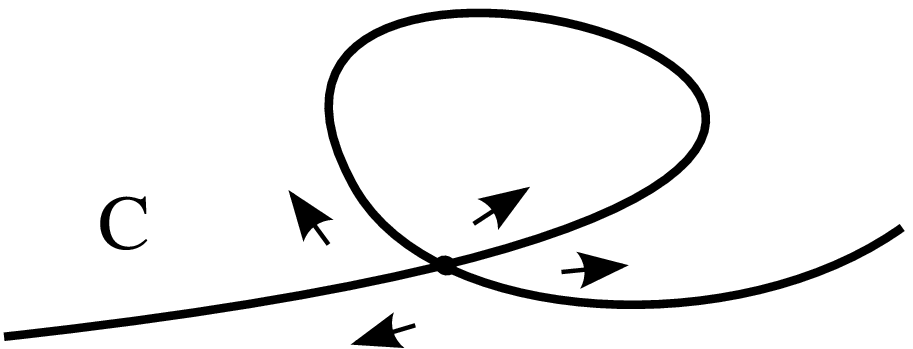}
  \caption{How the general fiber of a family in (S15) moves}
\end{figure}

To construct the family, start from $p_{2,3}\colon C\times C\times C\rightarrow C\times C$, blow up the diagonal $\Delta\subset C\times C\times C$ and then identify the proper transform of $\Delta_{1,2}:=p_{1,2}^*(\Delta)$ with the proper transform of $\Delta_{1,3}:=p_{1,3}^*(\Delta)$.
Then
\begin{eqnarray*}
\delta_0 &=&  -(\pi_1^*K_C + \pi_2^* K_C + 2\Delta)\\
\delta_1 &=& \Delta.
\end{eqnarray*}
The class $\kappa_2$ has been computed in \cite[\S 2.1 (1)]{MR1070600}. 
The curve $C$ has no generalized linear series with Brill-Noether number less than $0$, hence
\[
( 8 g^2 - 26 g + 20)A_{\kappa_1^2} +(2g-4)A_{\kappa_2} + (4-2g)A_{\delta_1^2} + 8(g-1)(g-2)A_{\delta_0^2} =0.
\]

\vskip1pc
\item[(S16)] For $\lfloor g/2 \rfloor \leq i \leq g-2$, take a general curve $C$ of genus $i$ and attach an elliptic curve $E$ and a general pointed curve $F$ of genus $g-i-1$ at two varying points in $C$. 

\begin{figure}[htbp]
\psfrag{C}[b][b]{$C$}
\psfrag{F}[b][b]{$F$}
\psfrag{E}[b][b]{$E$}
\centering
  \includegraphics[scale=0.5]{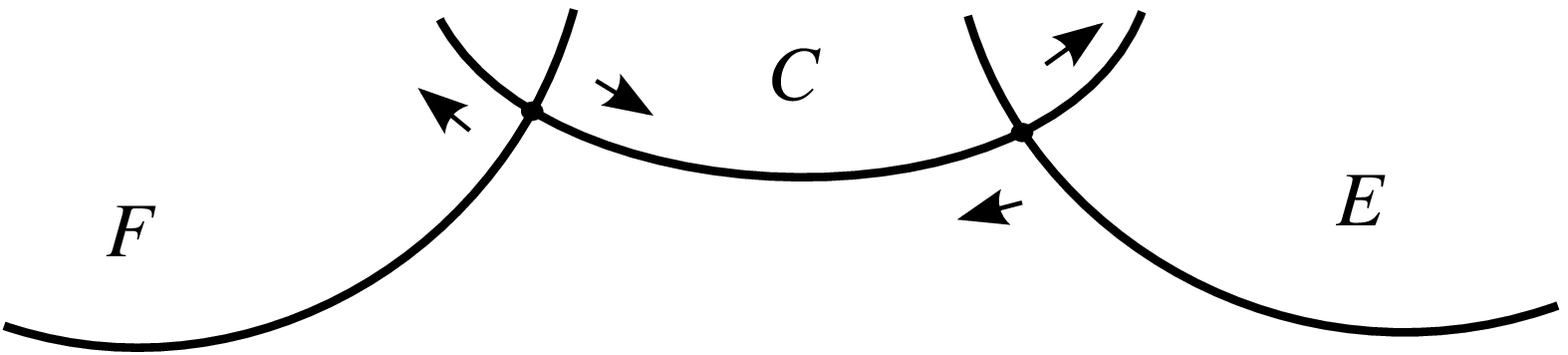}
  \caption{How the general fiber of a family in (S16) moves}
\end{figure}

To construct the family, blow up the diagonal $\Delta$ in $C\times C\times C$ as before, and then identify the proper transform of $\Delta_{1,2}$ with the zero section of a constant elliptic fibration over $C\times C$, and identify the proper transform of $\Delta_{1,3}$ with a general constant section of $F\times C\times C$.
For $i<g-2$
\begin{eqnarray*}
\delta_1 &=& -\pi_1^*K_C -\Delta \\
\delta_{g-i-1} &=& - \pi_2^* K_C - \Delta \\
\delta_i &=& \Delta
\end{eqnarray*}
while for $i=g-2$ the $\delta_1$ is the sum of the above $\delta_1$ and $\delta_{g-i-1}$.

Note that replacing the tail of genus $g-i-1$ with an elliptic tail does not affect the computation of the class $\kappa_2$, hence we can use the count from \cite[\S 3 ($\gamma$)]{MR1078265}, that is $\kappa_2 = 2i-2$. About the $\omega$ classes, on these test surfaces one has $\omega^{(i)}= - \delta_i^2$ and $\omega^{(i+1)}=-\delta^2_{i+1}=-\delta_{g-i-1}^2$. Finally note that $\delta_{1,g-i-1}$ is the product of the $c_1$'s coming from the two nodes, that is, $\delta_{1,g-i-1}=\delta_1\delta_{g-i-1}$.

If a fiber of this family has a $\g^1_k$ limit linear series $\{l_E,l_C,l_F \}$, then necessarily the adjusted Brill-Noether number has to be zero on $F$ and $E$, and $-2$ on $C$. Note that in any case $l_E=|2\cdot 0_E|$. From \S \ref{m} there are 
\[
\mathop{\sum_{\alpha=(\alpha_0,\alpha_1)}}_{\rho(i,1,k, \alpha) = -1} m_{i,k,\alpha}
\]
pairs in $C$ with such a property, $l_C$ is also uniquely determined and there are $N_{g-i-1,d,(d-1-\alpha_1,d-1-\alpha_0)}$ choices for $l_F$. With a similar argument to (S2), such pairs contribute with multiplicity one.

All in all for $i<g-2$
\begin{multline*}
(2i-2) \big[(4i-1)A_{\kappa_1^2} + A_{\kappa_2}+A_{\omega^{(i)}} -A_{\omega^{(i+1)}}
+A_{\delta_1^2}+(2i-1)A_{\delta_{1,g-i-1}}\big] \\
= \mathop{\sum_{0\leq \alpha_0\leq \alpha_1\leq k-1}}_{\alpha_0 + \alpha_1 = g-i-1} m_{i,k,(\alpha_0,\alpha_1)}\cdot  N_{g-i-1,k,(k-1-\alpha_1,k-1-\alpha_0)}
\end{multline*}
while for $i=g-2$
\begin{multline*}
(2g-6) \big[(4g-9)A_{\kappa_1^2} + A_{\kappa_2}+A_{\omega^{(g-2)}}
+(4g-8)A_{\delta_1^2}+(2g-5)A_{\delta_{1,1}}\big] \\
=m_{g-2,k,(0,1)}.
\end{multline*}

\vskip1pc
\item[(S17)] Consider a general element in $\theta_1$, vary the elliptic curve in a pencil of degree $12$ and vary one point on the elliptic curve.

\begin{figure}[htbp]
\psfrag{C}[b][b]{$C$}
\centering
  \includegraphics[scale=0.5]{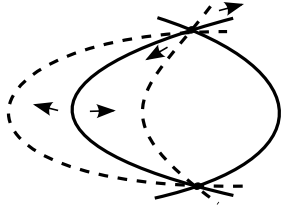}
  \caption{How the general fiber of a family in (S17) moves}
\end{figure}

The base of this family is the blow up of $\PP^2$ in the nine points of intersection of two general cubic curves. Let us denote by $H$ the pull-back of an hyperplane section in $\PP^2$, by $\Sigma$ the sum of the nine exceptional divisors and by $E_0$ one of them. We have
\begin{eqnarray*}
\lambda &=& 3H - \Sigma\\
\delta_0 &=& 30H - 10\Sigma -2E_0\\
\delta_1 &=& E_0
\end{eqnarray*}
(see also \cite[\S 2 (9)]{MR1023390}). Replacing the component of genus $g-2$ with a curve of genus $2$, we obtain a surface in $\Mbar_4$. The computation of the class $\kappa_2$ remains unaltered, that is $\kappa_2 = 1$ (see \cite[\S 3 ($\iota$)]{MR1078265}). Similarly for $\delta_{00}$ and $\theta_1$, while $\delta_{0,g-1}$ correspond to the value of $\delta_{01a}$ on the surface in $\Mbar_4$.

Let us study the intersection with $\Mbar^1_{2k,k}$. An admissible cover for some fiber of this family would necessarily have the two nodes in the same fiber, which is impossible, since the two points are general on the component of genus $g-2$. We deduce the following relation
\begin{eqnarray*}
3A_{\kappa_1^2} + A_{\kappa_2} -2 A_{\lambda\delta_0} +A_{\lambda\delta_1} -44 A_{\delta_0^2} -A_{\delta_1^2} + 12A_{\delta_{0,g-1}} -12A_{\delta_{00}} +A_{\theta_1}=0.
\end{eqnarray*}

\vskip1pc
\item[(S18)] 
For $2\leq i \leq \lfloor (g+1)/2 \rfloor $ we consider a general curve of type $\delta_{i-1,g-i}$ and we vary the central elliptic curve $E$ in a pencil of degree $12$ and one of the points on $E$. 

\begin{figure}[htbp]
\centering
  \includegraphics[scale=0.5]{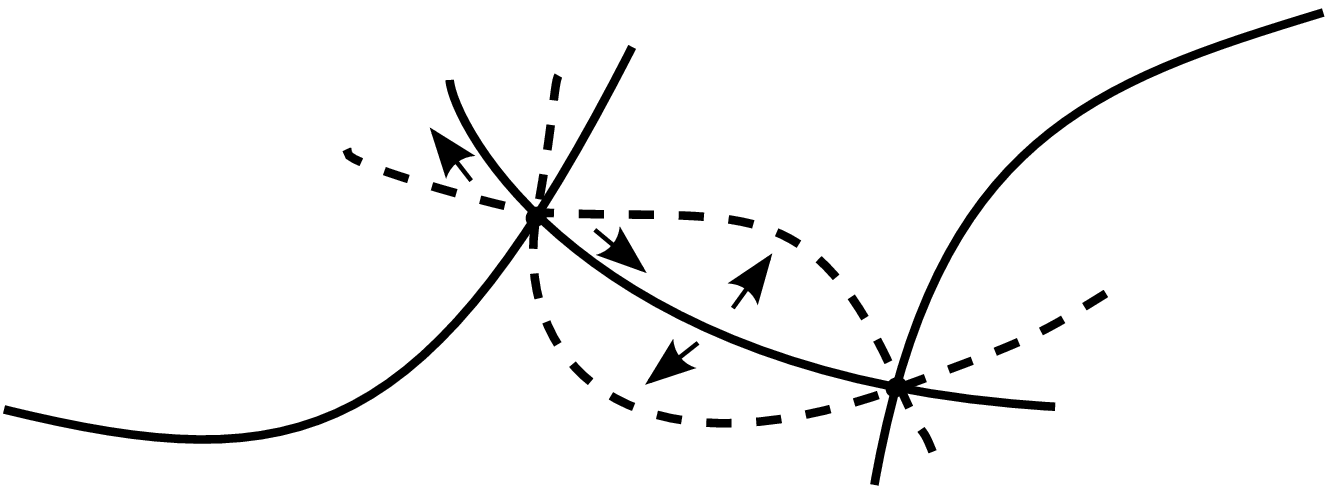}
  \caption{How the general fiber of a family in (S18) moves}
\end{figure}

The base of this family is the same surface as in (S17). For $i\geq 3$ we have
\begin{eqnarray*}
\lambda &=& 3H - \Sigma\\
\delta_0 &=& 12 \lambda\\
\delta_1 &=& E_0\\
\delta_{i-1} &=& -3H +\Sigma -E_0\\
\delta_{g-i} &=& -3H +\Sigma -E_0 
\end{eqnarray*}
while for $i=2$ the $\delta_1$ is the sum of the above $\delta_1$ and $\delta_{i-1}$, that is $\delta_1=-3H+\Sigma$ (see also \cite[\S 3 ($\lambda$)]{MR1078265}). 

Note that replacing the two tails of genus $i-1$ and $g-i$ with tails of genus $1$ and $2$, we obtain a surface in $\Mbar_4$. The computation of the class $\kappa_2$ remains unaltered, that is $\kappa_2 = 1$ (see \cite[\S 3 ($\lambda$)]{MR1078265}). Moreover, on these test surfaces $\omega^{(i)}= -\delta^2_{i}=- \delta_{g-i}^2$ and for $i\geq 3$, $\omega^{(g-i+1)}=-\delta^2_{g-i+1}=-\delta_{i-1}^2$ hold, while $\lambda^{(i)}=\lambda\delta_i=\lambda\delta_{g-i}$ for $i\geq 3$ and $\lambda^{(g-i+1)}=\lambda \delta_{g-i+1}=\lambda\delta_{i-1}$ for $i\geq 4$. 
All fibers are in $\delta_{i-1,g-i}$, hence $\delta_{i-1,g-i}$ is the product of the $c_1$'s of the two nodes, that is, $\delta_{i-1,g-i}=\delta_{i-1}\cdot\delta_{g-i}$. Note that on these surfaces, $\delta_{0,i-1}=\delta_0\delta_{i-1}$ and $\delta_{0,g-i}=\delta_0\delta_{g-i}$.
There are exactly $12$ fibers which contribute to $\theta_{i-1}$, namely when the elliptic curve degenerates into a rational nodal curve and the moving point hits the non-disconnecting node. Similarly, there are $12$ fibers which contribute to $\delta_{0,g-1}$, namely when the elliptic curve degenerates into a rational nodal curve and the moving point hits the disconnecting node.

These surfaces are disjoint from $\Mbar^1_{2k,k}$. Indeed the two tails of genus $i-1$ and $g-i$ have no linear series with adjusted Brill-Noether number less than $0$ at general points. Moreover an elliptic curve has no $\g^1_k$ with adjusted Brill-Noether number less than $-1$ at two arbitrary points. Finally a rational nodal curve has no generalized linear series with adjusted Brill-Noether number less than $0$ at arbitrary points.

It follows that for $i\geq 4$ we have
\begin{multline*}
3A_{\kappa_1^2}+A_{\kappa_2}-A_{\omega^{(i)}}-A_{\omega^{(g-i+1)}}-A_{\delta_1^2}+A_{\delta_{i-1,g-i}}-A_{\lambda^{(i)}}-A_{\lambda^{(g-i+1)}}\\
+A_{\lambda\delta_1} -12A_{\delta_{0,i-1}} -12A_{\delta_{0,g-i}} +12A_{\delta_{0,g-1}}+12A_{\theta_{i-1}}=0,
\end{multline*}
when $i=3$ 
\begin{multline*}
3A_{\kappa_1^2}+A_{\kappa_2}-A_{\omega^{(3)}}-A_{\omega^{(g-2)}}-A_{\delta_1^2}+A_{\delta_{2,g-3}}-A_{\lambda^{(3)}}-A_{\lambda\delta_2}\\
+A_{\lambda\delta_1} -12A_{\delta_{0,2}} -12A_{\delta_{0,g-3}} +12A_{\delta_{0,g-1}}+12A_{\theta_2}=0,
\end{multline*}
and when $i=2$
\begin{multline*}
3A_{\kappa_1^2}+A_{\kappa_2}-A_{\omega^{(2)}}+A_{\delta_{1,g-2}} \\
-A_{\lambda\delta_2}-12A_{\delta_{0,1}} -12A_{\delta_{0,g-2}} +12A_{\delta_{0,g-1}}+12A_{\theta_1}=0.
\end{multline*}

\end{enumerate} }

\section{Non-singularity}
\label{result}

In (S1)-(S18) we have constructed
\[
\lfloor (g^2-1)/4\rfloor + 3g-1
\]
linear relations in the coefficients $A$. Let us collect here all the relations.
For $2\leq i\leq \lfloor g/2\rfloor$ from (S1) we obtain
\[
2A_{\kappa_1^2} - A_{\omega^{(i)}} - A_{\omega^{(g-i)}} =\frac{T_i}{(2i-2)(2(g-i)-2)} , 
\]
from (S2) for $2\leq i\leq j \leq g-3$ and $i+j\leq g-1$
\[
2 A_{\kappa_1^2}+ A_{\delta_{ij}}  = \frac{D_{ij}}{(2i-2)(2j-2) },
\]
from (S3)
\[
4A_{\kappa_1^2} -A_{\omega^{(2)}}-A_{\omega^{(g-2)}}- A_{\delta_{1,g-2}} + 2 A_{\delta_{0,g-2}} = \frac{n_{g-2,k,(0,1)}}{g-3 },
\]
from (S4) for $2\leq i \leq g-3$
\[
4A_{\kappa_1^2} - A_{\delta_{1,i}} + 2 A_{\delta_{0,i}} +A_{\delta_{2,i}}  = \frac{D_{2,i}}{6(i-1) },
\]
from (S5)
\[
2A_{\kappa_1^2}-12 A_{\delta_{0,g-1}}+ 2A_{\delta_1^2} -A_{\lambda\delta_1} =0,
\]
from (S6) for $3\leq i \leq g-3$
\begin{eqnarray*}
2 A_{\kappa_1^2} - A_{\lambda^{(i)}} + A_{\delta_{1,g-i}} - 12 A_{\delta_{0,g-i}} = 0
\end{eqnarray*}
and
\[
2A_{\kappa_1^2} - A_{\lambda \delta_2}+ A_{\delta_{1,2}} - 12 A_{\delta_{0,2}} =0,
\]
from (S7)
\begin{multline*}
8A_{\kappa_1^2} + A_{\delta_{2,2}} -2A_{\delta_{1,2}} + A_{\delta_{1,1}} + 2A_{\delta_1^2}+ 8A_{\delta_0^2} + 4 A_{\delta_{00}} + 4 A_{\delta_{02}}-4A_{\delta_{01}} \\
= 4 N_{g-4,k,(0,1),(0,1)},
\end{multline*}
from (S8)
\[
2A_{\kappa_1^2} +288A_{\delta_0^2}+24A_{\lambda \delta_0}+2A_{\delta_1^2}-2A_{\lambda \delta_1}+ 144A_{\delta_{00}}+ A_{\delta_{1,1}} -24 A_{\delta_{01}}=0,
\]
from (S9) for $2\leq j \leq g-3$
\[
2 A_{\kappa_1^2}+ 2  A_{\delta_{1j}}+A_{\delta_{j,g-j-2}}-A_{\delta_{2,j}}-2A_{\delta_{j,g-j-1}}
- A_{\omega^{(j)}}-A_{\omega^{(g-j)}}  = 0,
\]
from (S10) for $g>6$
\begin{multline*}
2A_{\kappa_1^2}+12 A_{\delta_1^2} + 6 A_{\delta_{1,1}} +3A_{\delta_{1,g-5}}+2A_{\delta_{1,3}}+A_{\delta_{3,g-5}}+3A_{\delta_{1,g-3}}\\
-A_{\omega^{(3)}}-A_{\omega^{(g-3)}} - 3(A_{\delta_{2,g-3}} +A_{\delta_{2,g-5}}+2A_{\delta_{1,2}})\\
 -2(3A_{\delta_{1,g-4}}+A_{\delta_{3,g-4}}) - (3A_{\delta_{1,2}}+A_{\delta_{2,3}})
  +6A_{\delta_{2,g-4}}+ 3A_{\delta_{2,2}} =0
\end{multline*}
while for $g=6$
\begin{multline*}
2A_{\kappa_1^2}+18 A_{\delta_1^2} + 6 A_{\delta_{1,1}} +3A_{\delta_{1,g-5}}+2A_{\delta_{1,3}}+A_{\delta_{3,g-5}}+3A_{\delta_{1,g-3}}\\
-A_{\omega^{(3)}}-A_{\omega^{(g-3)}} - 3(A_{\delta_{2,g-3}} +A_{\delta_{2,g-5}}+2A_{\delta_{1,2}})\\
 -2(3A_{\delta_{1,g-4}}+A_{\delta_{3,g-4}}) - (3A_{\delta_{1,2}}+A_{\delta_{2,3}})
  +6A_{\delta_{2,g-4}}+ 3A_{\delta_{2,2}} =0,
\end{multline*}
from (S11)
\begin{multline*}
2 A_{\kappa_1^2} -A_{\lambda^{(g-3)}}+6A_{\delta_1^2}+3A_{\delta_{1,1}}-3 A_{\lambda \delta_1}+A_{\delta_{1,3}}-36A_{\delta_{01}}-12A_{\delta_{03}} \\
+3\left[A_{\lambda\delta_2}+12A_{\delta_{02}}-A_{\delta_{1,2}}\right] =0,
\end{multline*}
from (S12)
\begin{multline*}
2A_{\kappa_1^2}-3A_{\lambda \delta_1}-24A_{\delta_{01}}-12A_{\delta_{0,g-3}}-12A_{\delta_{0,g-1}}+6A_{\delta_1^2}+2A_{\delta_{1,1}}+A_{\delta_{1,g-3}}\\
-A_{\lambda^{(3)}} +2(A_{\lambda\delta_2}+12A_{\delta_{0,g-2}}-A_{\delta_{1,g-2}})+(A_{\lambda\delta_2}+12A_{\delta_{02}}-A_{\delta_{1,2}})=0,
\end{multline*}
from (S13)
\begin{multline*}
2(g-3)\left[4A_{\kappa_1^2} +2A_{\delta_{00}}+4A_{\delta_0^2}+A_{\delta_{02}} \right] +2A_{\delta_{0,g-2}}-A_{\omega^{(2)}}-A_{\omega^{(g-2)}}\\
-\left[2(g-3)A_{\delta_{01}}+A_{\delta_{1,g-2}}  \right] -\left[2A_{\delta_{01}}+A_{\delta_{1,2}}  \right] +\left[ A_{\delta_{1,1}}+2A_{\delta_1^2} \right] = 2  \cdot \ell_{g-2,k},
\end{multline*}
from (S14)
\begin{multline*}
(2g-4)\left[2A_{\kappa_1^2}-A_{\lambda \delta_0} -24A_{\delta_0^2}-12A_{\delta_{00}} +A_{\delta_{01}}\right]\\
-12A_{\delta_{0,g-1}} + (12A_{\delta_{01}}-A_{\delta_{1,1}})=0,
\end{multline*}
from (S15)
\[
( 8 g^2 - 26 g + 20)A_{\kappa_1^2} +(2g-4)A_{\kappa_2} + (4-2g)A_{\delta_1^2} + 8(g-1)(g-2)A_{\delta_0^2} =0,
\]
from (S16) for $\lfloor g/2 \rfloor \leq i \leq g-3$
\begin{multline*}
(4i-1)A_{\kappa_1^2} + A_{\kappa_2}+A_{\omega^{(i)}} -A_{\omega^{(i+1)}}
+A_{\delta_1^2}+(2i-1)A_{\delta_{1,g-i-1}} \\
= \frac{1}{2i-2} \mathop{\sum_{0\leq \alpha_0\leq \alpha_1\leq k-1}}_{\alpha_0 + \alpha_1 = g-i-1} m_{i,k,(\alpha_0,\alpha_1)}\cdot  N_{g-i-1,k,(k-1-\alpha_1,k-1-\alpha_0)}
\end{multline*}
and
\[
(4g-9)A_{\kappa_1^2} + A_{\kappa_2}+A_{\omega^{(g-2)}}
+(4g-8)A_{\delta_1^2}+(2g-5)A_{\delta_{1,1}} 
=\frac{m_{g-2,k,(0,1)}}{2g-6},
\]
from (S17)
\begin{eqnarray*}
3A_{\kappa_1^2} + A_{\kappa_2} -2 A_{\lambda\delta_0} +A_{\lambda\delta_1} -44 A_{\delta_0^2} -A_{\delta_1^2} + 12A_{\delta_{0,g-1}} -12A_{\delta_{00}} +A_{\theta_1}=0,
\end{eqnarray*}
from (S18) for $4\leq i\leq \lfloor (g+1)/2\rfloor$
\begin{multline*}
3A_{\kappa_1^2}+A_{\kappa_2}-A_{\omega^{(i)}}-A_{\omega^{(g-i+1)}}-A_{\delta_1^2}+A_{\delta_{i-1,g-i}}-A_{\lambda^{(i)}}-A_{\lambda^{(g-i+1)}}\\
+A_{\lambda\delta_1} -12A_{\delta_{0,i-1}} -12A_{\delta_{0,g-i}} +12A_{\delta_{0,g-1}}+12A_{\theta_{i-1}}=0,
\end{multline*}
and
\begin{multline*}
3A_{\kappa_1^2}+A_{\kappa_2}-A_{\omega^{(3)}}-A_{\omega^{(g-2)}}-A_{\delta_1^2}+A_{\delta_{2,g-3}}-A_{\lambda^{(3)}}-A_{\lambda\delta_2}\\
+A_{\lambda\delta_1} -12A_{\delta_{0,2}} -12A_{\delta_{0,g-3}} +12A_{\delta_{0,g-1}}+12A_{\theta_2}=0,
\end{multline*}
\begin{multline*}
3A_{\kappa_1^2}+A_{\kappa_2}-A_{\omega^{(2)}}+A_{\delta_{1,g-2}} \\
-A_{\lambda\delta_2}-12A_{\delta_{0,1}} -12A_{\delta_{0,g-2}} +12A_{\delta_{0,g-1}}+12A_{\theta_1}=0.
\end{multline*}

\vskip1pc

Our aim is to show that the above linear relations yield a non-degenerate linear system. Let 
\begin{multline*}
\{e_{\kappa_1^2}, e_{\kappa_2}, e_{\delta_0^2}, e_{\lambda\delta_0}, e_{\delta_1^2}, e_{\lambda\delta_1}, e_{\lambda\delta_2}, \dots, e_{\omega^{(i)}},\dots, e_{\lambda^{(j)}}, \dots,\\
\dots, e_{\delta_{k,l}}, \dots, e_{\theta_1},\dots,e_{\theta_{\lfloor (g-1)/2\rfloor}}\}
\end{multline*}
be the canonical basis of $\QQ^{\lfloor (g^2-1)/4\rfloor + 3g-1}$ indexed by the tautological codimension-two generating classes from \S \ref{generatingclasses}.
Let $Q_g$ be the square matrix of order $\lfloor (g^2-1)/4\rfloor + 3g-1$ associated to the linear system given by the above relations. 
As we have already noted, since the test surfaces in (S1)-(S18) are defined also for odd values of $g\geq 6$, the matrix $Q_g$ is defined also for $g$ odd, $g\geq 7$. For each $g\geq 6$ we construct a square matrix $T_g$ of order $\lfloor (g^2-1)/4\rfloor + 3g-1$ such that $Q_g \cdot T_g$ is lower-triangular with nonzero diagonal coefficients.

We describe the columns of $T_g$ dividing them into $18$ groups, similarly to the description of the relations that yield the rows of $Q_g$. 

\vskip1pc

{\setlength{\leftmargini}{0pt} 
\begin{enumerate}
\itemsep0.5em 
\item[(T1)] For $2\leq i \leq \lfloor g/2 \rfloor$ consider the column $e_{\omega^{(i)}}$.

\item[(T2)] For $i,j$ such that $2\leq i\leq j \leq g-3$ and $i+j\leq g-1$ consider $e_{\delta_{i,j}}$.

\item[(T3)] Consider the column $e_{\delta_{1,g-2}}$.

\item[(T4)] For $2\leq i \leq g-3$ consider $e_{\delta_{1,i}}$.

\item[(T5)] Take $e_{\delta_{0,g-1}}$.

\item[(T6)] For $3\leq i \leq g-3$ consider $e_{\lambda^{(i)}}$. Moreover, consider the column $e_{\lambda\delta_2}$.

\item[(T7)] Consider the column $e_{\delta_{1,1}}$.

\item[(T8)] Consider $e_{\lambda\delta_0}$.

\item[(T9)] Consider 
\[
2e_{\delta_{1,2}}+e_{\delta_{0,2}}-10e_{\lambda\delta_2}.
\]
Moreover for $3\leq j \leq g-3$ consider 
\[
2e_{\delta_{1,j}}+e_{\delta_{0,j}}-10e_{\lambda^{(g-j)}}.
\]

\noindent Take the following:

\vskip1pc

\item[(T10)] \hfill
$\begin{aligned}[t]
60e_{\lambda\delta_1}+12e_{\delta_1^2}-3e_{\delta_{0,g-1}}+8e_{\delta_{0,1}}+2e_{\delta_{00}}
\end{aligned}
$\hfill\null

\item[(T11)]  
\hfill
$\begin{aligned}[t]
12e_{\lambda\delta_1}+e_{\lambda\delta_0}-e_{\delta_{0,g-1}}
\end{aligned}
$\hfill\null

\item[(T12)] 
\hfill
$\begin{aligned}[t]
e_{\delta_{0,g-2}}+2e_{\delta_{1,g-2}}
\end{aligned}
$\hfill\null

\item[(T13)] 
\hfill
$\begin{aligned}[t]
12e_{\lambda\delta_1}+6e_{\lambda\delta_0}-e_{\delta_{0,g-1}}-e_{\delta_{0,1}}-e_{\delta_{00}}
\end{aligned}
$\hfill\null

\item[(T14)] 
\hfill
$\begin{aligned}[t]
6\Bigg(e_{\kappa_1^2}+e_{\omega^{(\lfloor g/2 \rfloor)}}+ 2\sum_{2\leq s<g/2} e_{\omega^{(s)}}+12 \left(e_{\lambda\delta_0}+2e_{\lambda\delta_1}+2e_{\lambda\delta_2}+2\sum _{3\leq s \leq g-3} e_{\lambda^{(s)}} \right)  \\
 -2\sum_{(i,j)\not= (0,g-1)} e_{\delta_{ij}}  \Bigg) -11e_{\delta_{0,g-1}}
\end{aligned}
$\hfill\null

\item[(T15)]  
\hfill
$\begin{aligned}[t]
e_{\kappa_2}.
\end{aligned}
$\hfill\null

\item[(T16)] For $\lfloor g/2 \rfloor \leq i \leq g-3$ consider $e_{\omega^{(i+1)}}-e_{\omega^{(g-i-1)}}$. Furthermore consider 
\begin{multline*}
(g-1)\left(  \sum_{2\leq s\leq g/2}12\left(\frac{g}{2}-s\right) (e_{\omega^{(g-s)}} - e_{\omega^{(s)}})+12e_{\delta_1^2} -24e_{\delta_{1,1}}+2e_{\delta_{0,g-1}} \right)\\+3e_{\delta_0^2}-6e_{\delta_{00}}.
\end{multline*}

\item[(T17)] Consider 
\begin{multline*}
 6g\, e_{\kappa_2}+ \sum_{2\leq s\leq g/2}12\left(\frac{g}{2}-s\right) (e_{\omega^{(g-s)}} - e_{\omega^{(s)}})+12\left(1-\frac{g}{2}\right)e_{\delta_1^2} \\
 +12(g-2)e_{\delta_{1,1}}-3e_{\delta_0^2}+(2-g)e_{\delta_{0,g-1}} +6e_{\delta_{00}}.
\end{multline*}

\item[(T18)] For $4\leq i \leq \lfloor (g+1)/2 \rfloor $ consider $e_{\theta_{i-1}}$. 
Moreover consider $e_{\theta_2}$ and finally the column
\begin{multline*}
{}-6g\, e_{\kappa_2}+ \sum_{2\leq s\leq g/2}12\left(\frac{g}{2}-s\right) (e_{\omega^{(s)}}-e_{\omega^{(g-s)}})+12\left(\frac{g}{2}-1\right)e_{\delta_1^2} \\
 +12(2-g)e_{\delta_{1,1}}+3e_{\delta_0^2}+(g-2)e_{\delta_{0,g-1}} -6e_{\delta_{00}}+72e_{\theta_1}.
\end{multline*}

\end{enumerate}

\vskip1pc

One checks that $Q_g\cdot T_g$ is an lower-triangular matrix with all the coefficients on the main diagonal different from zero.
It follows that $\det(Q_g)\not= 0$ for all $g\geq 6$. 
In particular, when $g=2k$, we are able to solve the system and find the coefficients $A$.

\begin{thm}
\label{BNcodim2mainthm}
For $k \geq 3$, the class of the locus $\Mbar^1_{2k,k}\subset\Mbar_{2k}$ is
\begin{eqnarray*}
\left[ \Mbar^1_{2k,k} \right]_Q &\!\!\! = \!\!\!& c\Bigg[ A_{\kappa_1^2}\kappa_1^2 + A_{\kappa_2}\kappa_2 + A_{\delta_0^2}\delta_0^2 + A_{\lambda\delta_0}\lambda\delta_0 + A_{\delta_1^2}\delta_1^2  + A_{\lambda\delta_1}\lambda\delta_1  \nonumber\\
&& {} + A_{\lambda\delta_2}\lambda\delta_2 + \sum_{i=2}^{2k-2} A_{\omega^{(i)}}\omega^{(i)} + \sum_{i=3}^{2k-3} A_{\lambda^{(i)}}\lambda^{(i)} + \sum_{i,j} A_{\delta_{ij}}\delta_{ij} \nonumber\\
&&{}+ \sum_{i=1}^{\lfloor (2k-1)/2 \rfloor} A_{\theta_i}\theta_{i} \Bigg]
\end{eqnarray*}
in $R^2(\Mbar_{2k})$, where
\begin{eqnarray*}
c &=&  \frac{2^{k-6}(2k-7)!!}{3(k!)}\\
A_{\kappa_1^2} = - A_{\delta_0^2}&=&  3 k^2+3 k+5 \\
A_{\kappa_2} &=&  -24k(k+5)\\
A_{\delta_1^2} &=& -(3 k (9 k+41)+5) \\
A_{\lambda\delta_0} &=& -24(3 (k-1) k-5) \\
A_{\lambda\delta_1} &=& 24 \left(-33 k^2+39 k+65\right) \\
A_{\lambda\delta_2} &=& 24 (3 (37-23 k) k+185)\\
A_{\omega^{(i)}} &=& -180 i^4+120 i^3 (6 k+1)-36 i^2 \left(20 k^2+24 k-5\right)\\
   && {}+24 i \left(52 k^2-16 k-5\right)+27 k^2+123 k+5 \\
A_{\lambda^{(i)}} &=& 24 [6i^2 (3 k+5)-6 i \left(6 k^2+23 k+5\right)\\
			&& {} +159 k^2+63 k+5]\\
A_{\theta^{(i)}} &=&  -12 i[ 5 i^3+i^2 (10-20k)+i \left(20 k^2-8k-5\right)\\
			&& {}-24 k^2+32k-10] \\	
A_{\delta_{1,1}} &=& 48 \left(19 k^2-49 k+30\right)\\
A_{\delta_{1,2k-2}} &=&  \frac{2}{5} (3 k (859 k-2453)+2135)\\
A_{\delta_{00}} &=& 24k(k-1) \\
A_{\delta_{0,2k-2}} &=& \frac{2}{5} (3 k (187 k-389)-745)\\
A_{\delta_{0,2k-1}} &=& 2 (k (31 k-49)-65)
\end{eqnarray*}
and for $i\geq 1$ and $2\leq j \leq 2k-3$
\begin{eqnarray*}
A_{\delta_{ij}} &=&  	2 [3k^2 (144 i j-1)-3k (72 i j (i+j+4)+1)\\
			&&{}+180 i (i+1) j (j+1)-5] 
\end{eqnarray*}
while
\begin{eqnarray*}
A_{\delta_{0j}} &=& 2 \left(-3 \left(12 j^2+36 j+1\right)k+(72 j-3) k^2-5\right)
\end{eqnarray*}
for $1\leq j \leq 2k-3$.
\end{thm}

As usual, for a positive integer $n$, the symbol $(2n+1)!!$ denotes 
\[
\frac{(2n+1)!}{2^n \cdot n!},
\]
while $(-1)!!=1$.

\section{Pull-back to \texorpdfstring{$\Mbar_{2,1}$}{Mbar{2,1}}}
\label{pullbacktoM2,1}

As a check, in this section we obtain four more relations for the coefficients $A$ considering the pull-back of $\Mbar^1_{2k,k}$ to $\Mbar_{2,1}$. Let $j\colon \Mbar_{2,1}\rightarrow \Mbar_g$ be the map obtained by attaching at elements $(D,p)$ in $\Mbar_{2,1}$ a fixed general pointed curve of genus $g-2$. This produces a map $j^*\colon A^2(\Mbar_g) \rightarrow A^2(\Mbar_{2,1})$.

In \cite[Chapter 3 \S 1]{faberthesis} it is shown that $A^2(\Mbar_{2,1})$ has rank $5$ and is generated by the classes of the loci composed by curves of type $\Delta_{00}, (a), (b), (c)$ and $(d)$ as in Fig. \ref{lociinM21}. 

\begin{figure}[htbp]
\psfrag{D}[b][b]{$\Delta_{00}$}
\psfrag{a}[b][b]{$(a)$}
\psfrag{b}[b][b]{$(b)$}
\psfrag{c}[b][b]{$(c)$}
\psfrag{d}[b][b]{$(d)$}
\centering
  \includegraphics[scale=0.5]{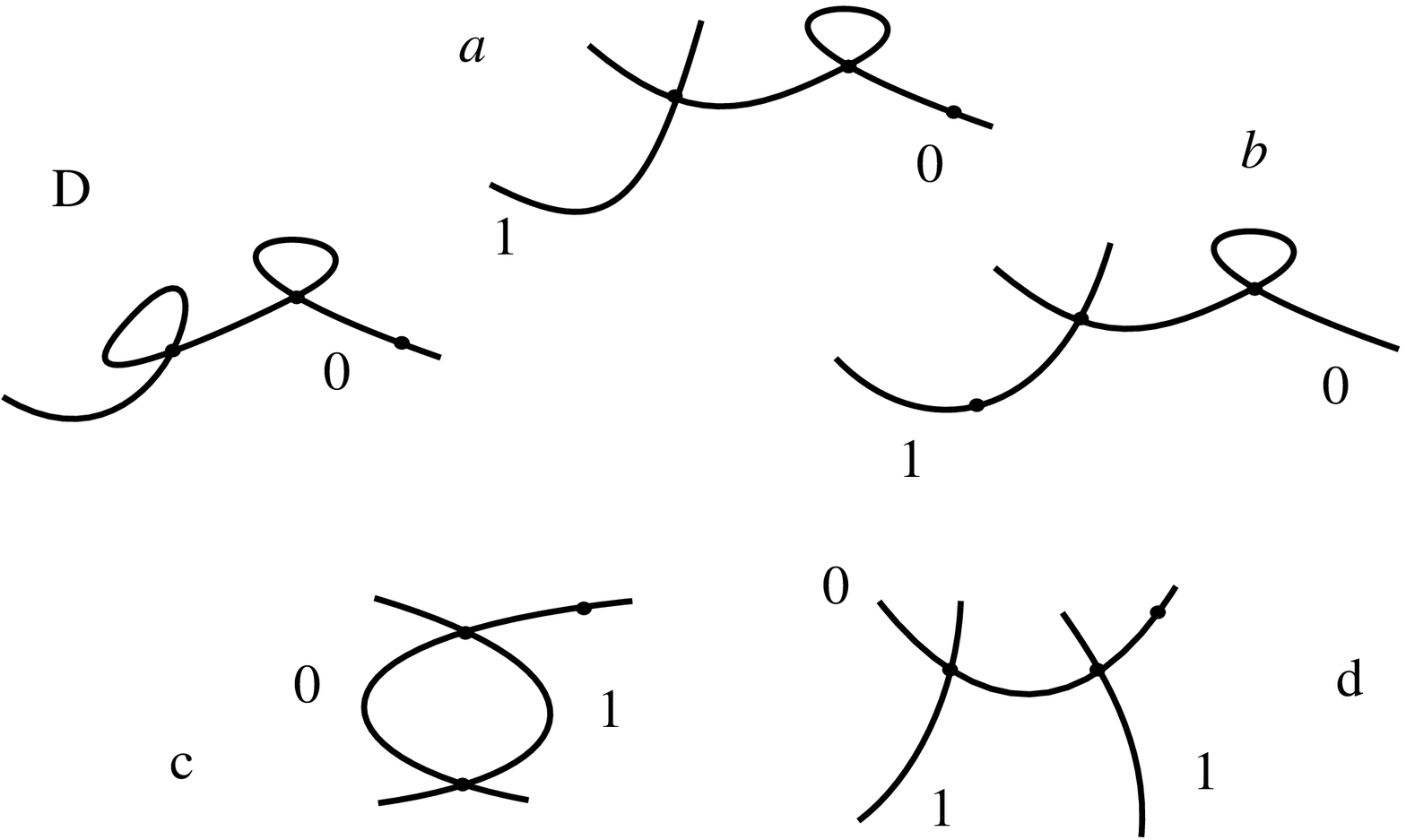}
  \caption{Loci in $\Mbar_{2,1}$}
  \label{lociinM21}
\end{figure}

We have the following pull-backs
\begin{eqnarray*}
j^*(\delta_{0,1}) &=&[(a)]_Q \\
j^*(\delta_{0,g-1}) &=& [(b)]_Q\\
j^*(\theta_1) &=& [(c)]_Q\\
j^*(\delta_{1,1}) & = & [(d)]_Q\\
j^*(\delta_{00}) &=& 	[\Delta_{00}]_Q\\
j^*(\delta_0^2) &=&  \frac{5}{3}[\Delta_{00}]_Q -2[(a)]_Q-2[(b)]_Q\\
j^*(\delta_1^2) &=& -\frac{1}{12}\left([(a)]_Q+[(b)]_Q\right)\\
j^*(\lambda \delta_0) &=& \frac{1}{6}[\Delta_{00}]_Q\\
j^*(\lambda \delta_1) &=& \frac{1}{12}\left([(a)]_Q+[(b)]_Q\right)\\
j^*(\lambda \delta_2) & = & -\lambda \psi\\
		      &=& \frac{1}{60}(-[\Delta_{00}]_Q -7 [(a)]_Q  - 12 [(c)]_Q -24 [(d)]_Q )\\								
j^*(\kappa_1^2) &=& \left( \frac{1}{5}\delta_0 + \frac{7}{5}\delta_1+ \psi \right)^2 \\
                          &=& \frac{1}{120}(17 [\Delta_{00}]_Q + 127 [(a)]_Q + 37 [(b)]_Q + 120 [(c)]_Q \\
                          &&{} + 840 [(d)]_Q)
\end{eqnarray*}
\begin{eqnarray*}	                          
j^*(\kappa_2) &=& \lambda (\lambda + \delta_1) + \psi^2\\
		      &=&  \frac{1}{120}(3 [\Delta_{00}]_Q + 25 [(a)]_Q + 11 [(b)]_Q + 24 [(c)]_Q + 168 [(d)]_Q )\\
j^*(\delta_{1,g-2}) &=& -\delta_1 \psi\\
			    &=& -\frac{1}{12}[(a)]_Q -2[(d)]_Q\\
j^*(\delta_{0,g-2}) &=& -\delta_0 \psi\\
			    &=& -\frac{1}{6}[\Delta_{00}]_Q-[(a)]_Q-2[(c)]_Q\\			    		      
j^*(\omega^{(2)}) &=& - \psi^2\\
			    &=& -\frac{1}{120}([\Delta_{00}]_Q +13 [(a)]_Q - [(b)]_Q + 24 [(c)]_Q + 168 [(d)]_Q ).
\end{eqnarray*}

For this, see relations in \cite[Chapter 3 \S 1]{faberthesis} and \cite[\S8 - \S 10]{MR717614}. We have used that on $\Mbar_{g,1}$ one has 
\[
\kappa_i=\kappa_i|_{\Mbar_g}+\psi^i
\]
(see \cite[1.10]{MR1486986}). 

All the other classes have zero pull-back. Finally, $j^*(\Mbar^1_{2k,k})$ is supported at most on the locus $(c)$. Indeed general elements in the loci $\Delta_{00}, (a), (b)$ and $(d)$ does not admit any linear series $\g^1_k$ with adjusted Brill-Noether number less than $-1$ (see also \cite[Lemma 5.1]{MR1185606}). Since the restriction of $\Mbar^1_{2k,k}$ to $j(\Mbar_{2,1})$ is supported in codimension two, then $j(\Mbar_{2,1}\setminus (c))=0$. Hence looking at the coefficient of $[\Delta_{00}]_Q$, $[(a)]_Q$, $[(b)]_Q$ and $[(d)]_Q$ 
in $j^*(\Mbar^1_{2k,k})$ we obtain the following relations
\begin{multline*}
A_{\delta_{00}}+\frac{5}{3}A_{\delta_0^2}+\frac{1}{6}A_{\lambda\delta_0}-\frac{1}{60}A_{\lambda\delta_2}+\frac{17}{120}A_{\kappa_1^2}+\frac{1}{40}A_{\kappa_2}
{}-\frac{1}{6}A_{\delta_{0,g-2}}-\frac{1}{120}A_{\omega^{(2)}}=0\\
A_{\delta_{01}}-2A_{\delta_0^2}-\frac{1}{12}A_{\delta_1^2}+\frac{1}{12}A_{\lambda\delta_1}-\frac{7}{60}A_{\lambda\delta_2}+\frac{127}{120}A_{\kappa_1^2}\\
{}+\frac{5}{24}A_{\kappa_2}-\frac{1}{12}A_{\delta_{1,g-2}}-A_{\delta_{0,g-2}}-\frac{13}{120}A_{\omega^{(2)}}=0
\\
A_{\delta_{0,g-1}}-2A_{\delta_0^2}-\frac{1}{12}A_{\delta_1^2}+\frac{1}{12}A_{\lambda\delta_1}+\frac{37}{120}A_{\kappa_1^2}+\frac{11}{120}A_{\kappa_2}+\frac{1}{120}A_{\omega^{(2)}}=0\\
A_{\delta_{1,1}}-\frac{2}{5}A_{\lambda\delta_2}+7A_{\kappa_1^2}+\frac{7}{5}A_{\kappa_2}-2A_{\delta_{1,g-2}}-\frac{7}{5}A_{\omega^{(2)}}=0.
\end{multline*}
The coefficients $A$ shown in Thm. \ref{BNcodim2mainthm} satisfy these relations.

\section{Further relations}

In this section we will show how to get further relations for the coefficients $A$ that can be used to produce more tests for our result.

\subsection{The coefficients of $\kappa_1^2$ and $\kappa_2$}

One can compute the class of $\M^1_{2k,k}$ in the open $\M_{2k}$ by the methods described by Faber in \cite{MR1722541}. Let $\mathcal{C}^k_{2k}$ be the $k$-fold fiber product of the universal curve over $\M_{2k}$ and let $\pi_{i}\colon \mathcal{C}^k_{2k} \rightarrow \mathcal{C}_{2k}$ be the map forgetting all but the $i$-th point, for $i=1,\dots,k$. We define the following tautological classes on $\mathcal{C}^k_{2k}$: $K_i$ is the class of $\pi_i^*(\omega)$, where $\omega$ is the relative dualizing sheaf of the map $\mathcal{C}_{2k}\rightarrow \M_{2k}$, and $\Delta_{i,j}$ is the class of the locus of curves with $k$ points $(C,x_1,\dots,x_k)$ such that $x_i=x_j$, for $1\leq i, j \leq k$.

Let $\mathbb{E}$ be the pull-back to $\mathcal{C}^k_{2k}$ of the Hodge bundle of rank $2k$ and let $\mathbb{F}_k$ be the bundle on  
$\mathcal{C}^k_{2k}$ of rank $k$ whose fiber over $(C,x_1,\dots,x_k)$ is 
\[
H^0(C, K_C/ K_C(-x_1\cdots-x_k)).
\]
We consider the locus $X$ in $\mathcal{C}^k_{2k}$ where the evaluation map
\[
\varphi_k \colon \mathbb{E}\rightarrow \mathbb{F}_k
\]
has rank at most $k-1$. Equivalently, $X$ parameterizes curves with $k$ points $(C,x_1,\dots,x_k)$ such that $H^0(C,K_C(-x_1\cdots-x_k))\geq k+1$ or, in other terms, $H^0(C, x_1+\cdots+x_k)\geq 2$. By Porteous formula, the class of $X$ is
\[
[X]= \left[
\begin{array}{ccccc}
e_1 & e_2 & e_3 & \cdots &  e_{k+1}\\
1     & e_1 &  e_2 & \cdots	& e_{k}\\
0	&  1	& e_1 &	\ddots & \vdots\\
\vdots	&  \ddots & \ddots	&  \ddots & e_2\\
0 & \cdots & 0&	1&	e_1
\end{array}
\right]
\]
\noindent where the $e_i$'s are the Chern classes of $\mathbb{F}_k - \mathbb{E}$. The total Chern class of $\mathbb{F}_k - \mathbb{E}$ is
\begin{multline*}
(1+K_1)(1+K_2-\Delta_{1,2})\cdots(1+K_k-\Delta_{1,k}\cdots-\Delta_{k-1,k})(1-\lambda_1+\lambda_2-\lambda_3+\cdots+\lambda_{2k}).
\end{multline*}
Intersecting the class of $X$ with $\Delta_{1,2}$ we obtain a class that pushes forward via $\pi:=\pi_1\pi_2\cdots \pi_k$ to the class of $\M^1_{2k,k}$ with multiplicity $(k-2)!(6k-2)$. We refer the reader to \cite[\S 4]{MR1722541} for formulae for computing the push-forward $\pi_*$.

For instance, when $k=3$ one constructs a degeneracy locus $X$ on the $3$-fold fiber product of the universal curve over $\M_6$.
The class of $X$ is
\[
[X] = e_1^4 - 3e_1^2 e_2 + e_2^2 + 2e_1 e_3 - e_4
\]
where the $e_i$'s are determined by the following total Chern class
\begin{eqnarray*}
(1+K_1)(1+K_2-\Delta_{1,2})(1+K_3-\Delta_{1,3}-\Delta_{2,3})(1-\lambda_1+\lambda_2-\lambda_3+\cdots+\lambda_{6}).
\end{eqnarray*}
Upon intersecting the class of $X$ with $\Delta_{1,2}$ and using the following identities
\begin{eqnarray*}
&\Delta_{1,3}\Delta_{2,3}= \Delta_{1,2}\Delta_{1,3}&\\
\Delta_{1,2}^2=-K_1\Delta_{1,2} & \Delta_{1,3}^2=-K_1\Delta_{1,3}&\Delta_{2,3}^2=-K_2\Delta_{2,3}\\
K_2\Delta_{1,2}=K_1\Delta_{1,2}&K_3 \Delta_{1,3}=K_1\Delta_{1,3}&K_3\Delta_{2,3}=K_2\Delta_{2,3},
\end{eqnarray*}
one obtains
\begin{eqnarray*}
[X]\cdot \Delta_{1,2} &=&   K_3^4 \Delta_{1,2}-3  K_3^3 \Delta_{1,2}^2+7  K_3^2 \Delta_{1,2}^3-15  K_3 \Delta_{1,2}^4+31 \Delta_{1,2}^5\\
&&{}+72 \Delta_{1,2} \Delta_{2,3}^4+172 \Delta_{1,3} \Delta_{2,3}^4- K_3^3 \Delta_{1,2} \lambda_1+3  K_3^2 \Delta_{1,2}^2 \lambda_1\\
&&{}-7  K_3 \Delta_{1,2}^3 \lambda_1+15 \Delta_{1,2}^4 \lambda_1+23 \Delta_{1,2} \Delta_{2,3}^3 \lambda_1+41 \Delta_{1,3} \Delta_{2,3}^3 \lambda_1\\
&&{}+ K_3^2 \Delta_{1,2} \lambda_1^2-3  K_3 \Delta_{1,2}^2 \lambda_1^2+7 \Delta_{1,2}^3 \lambda_1^2+6 \Delta_{1,2} \Delta_{2,3}^2 \lambda_1^2\\
&&{}+8 \Delta_{1,3} \Delta_{2,3}^2 \lambda_1^2- K_3 \Delta_{1,2} \lambda_1^3+3 \Delta_{1,2}^2 \lambda_1^3+\Delta_{1,2} \Delta_{2,3} \lambda_1^3\\
&&{}+\Delta_{1,3} \Delta_{2,3} \lambda_1^3+\Delta_{1,2} \lambda_1^4- K_3^2 \Delta_{1,2} \lambda_2+3  K_3 \Delta_{1,2}^2 \lambda_2-7 \Delta_{1,2}^3 \lambda_2\\
&&{}-6 \Delta_{1,2} \Delta_{2,3}^2 \lambda_2-8 \Delta_{1,3} \Delta_{2,3}^2 \lambda_2+2  K_3 \Delta_{1,2} \lambda_1 \lambda_2-6 \Delta_{1,2}^2 \lambda_1 \lambda_2\\
&&{}-2 \Delta_{1,2} \Delta_{2,3} \lambda_1 \lambda_2-2 \Delta_{1,3} \Delta_{2,3} \lambda_1 \lambda_2-3 \Delta_{1,2} \lambda_1^2 \lambda_2+\Delta_{1,2} \lambda_2^2\\
&&{}- K_3 \Delta_{1,2} \lambda_3+3 \Delta_{1,2}^2 \lambda_3+\Delta_{1,2} \Delta_{2,3} \lambda_3+\Delta_{1,3} \Delta_{2,3} \lambda_3\\
&&{}+2 \Delta_{1,2} \lambda_1 \lambda_3-\Delta_{1,2} \lambda_4.
\end{eqnarray*}
Computing the push-forward to $\M_6$ of the above class, one has
\begin{eqnarray*}
\left[ \M^1_{6,3} \right]_Q &= & \frac{1}{16} \big( ( 18 \kappa_0 -244) \kappa_2 + 7\kappa_1^2 +(64-10\kappa_0)\kappa_1 \lambda_1 \\
					&&{} + (3\kappa_0^2-14\kappa_0) \lambda_1^2 + (14\kappa_0-3\kappa_0^2)\lambda_2 \big).
\end{eqnarray*}
Note that $\kappa_0 = 2g-2=10$, $12\lambda_1=\kappa_1$ and $2\lambda_2=\lambda_1^2$, hence we recover
\begin{eqnarray*}
\left[ \M^1_{6,3} \right]_Q &= & \frac{41}{144}\kappa_1^2 -4 \kappa_2.
\end{eqnarray*}

\begin{rem}
As a corollary one obtains the class of the Maroni locus in $\M_6$.
The trigonal locus in $\M_{2k}$ has a divisor known as the Maroni locus (see \cite{MR0024182}, \cite{MR882414}). While the general trigonal curve of even genus admits an embedding in $\PP^1 \times \PP^1$ or in $\PP^2$ blown up in one point, the trigonal curves admitting an embedding to other kind of ruled surfaces constitute a subvariety of codimension one inside the trigonal locus.

The class of the Maroni locus in the Picard group of the trigonal locus in $\Mbar_{2k}$ has been studied in \cite{MR1775309}. For $k=3$, one has that the class of the Maroni locus is $8\lambda\in \mbox{Pic}_\QQ(\Mbar^1_{6,3})$. Knowing the class of the trigonal locus in $\M_6$, one has that the class of the Maroni locus in $\M_6$ is
\[
8\lambda \left( \frac{41}{144}\kappa_1^2 -4 \kappa_2\right)  .
\]
\end{rem}

\subsection{More test surfaces}

One could also consider more test surfaces. For instance one can easily adapt the test surfaces of type $(\varepsilon)$ and $(\kappa)$ from \cite[\S 3]{MR1078265}. They are all disjoint from the locus $\Mbar^1_{2k,k}$ and produce relations compatible with the ones we have shown.

\subsection{The relations for $g=5$} As an example, let us consider the case $g=5$. We know that the tautological ring of $\M_5$ is generated by $\lambda$, that is, there is a non-trivial relation among $\kappa_1^2$ and $\kappa_2$ (see \cite{MR1722541}). The square matrix $Q_5$ from  \S \ref{result} expressing the restriction of the generating classes in $\Mbar_5$ to the test surfaces (S1)-(S18) (we have to exclude the relation from (S10) which is defined only for $g\geq 6$), has rank $19$, hence showing that the class $\kappa_1^2$ (or the class $\kappa_2$) and the $18$ boundary classes in codimension two in $\Mbar_5$ are independent.

\section{The hyperelliptic locus in \texorpdfstring{$\Mbar_4$}{Mbar4}}

The class of the hyperelliptic locus in $\Mbar_4$ has been computed in \cite[Prop. 5]{MR2120989}. In this section we will recover the formula by the means of the techniques used so far.

The class will be expressed as a linear combination of the $14$ generators for $R^2(\Mbar_4)$ from \cite{MR1078265}: $\kappa_2$, $\lambda^2$, $\lambda\delta_0$, $\lambda\delta_1$, $\lambda\delta_2$, $\delta_0^2$, $\delta_0\delta_1$, $\delta_1^2$, $\delta_1\delta_2$, $\delta_2^2$, $\delta_{00}$, $\gamma_1$, $\delta_{01a}$ and $\delta_{1,1}$. Remember that there exists one unique relation among these classes, namely
\begin{multline*}
60 \kappa_2 - 810 \lambda^2 + 156 \lambda\delta_0 + 252 \lambda \delta_1 -3\delta_0^2 - 24\delta_0\delta_1\\
{}+24\delta_1^2 -9\delta_{00}+7\delta_{01a} -12\gamma_1-84\delta_{1,1}=0,
\end{multline*}
hence $R^2(\Mbar_4)$ has rank $13$. Write $[\Mbar^1_{4,2}]_Q$ as
\begin{eqnarray*}
 \left[ \Mbar^1_{4,2} \right]_Q & = & A_{\kappa_2}\kappa_2 + A_{\lambda^2}\lambda^2+ A_{\lambda\delta_0}\lambda\delta_0+ A_{\lambda\delta_1}\lambda\delta_1+ A_{\lambda\delta_2}\lambda\delta_2+ A_{\delta_0^2}\delta_0^2\\
&& {}+ A_{\delta_0\delta_1}\delta_0\delta_1 + A_{\delta_1^2}\delta_1^2 + A_{\delta_1\delta_2}\delta_1\delta_2 +A_{\delta_2^2}\delta_2^2+ A_{\delta_{00}}\delta_{00}+ A_{\gamma_1}\gamma_1\\
&&{}+ A_{\delta_{01a}}\delta_{01a}+ A_{\delta_{1,1}}\delta_{1,1}.
\end{eqnarray*}
Let us construct $13$ independent relations among the coefficients $A$.

The surfaces (S1), (S3), (S5), (S6), (S8), (S12)-(S18) from \S \ref{testsurfaces} give respectively the following $12$ independent relations
\begin{eqnarray*}
8 A_{\delta_2^2}=36\\
4A_{\delta_2^2} -2A_{\delta_1 \delta_2}= 12\\
-4A_{\lambda\delta_1}-48A_{\delta_0\delta_1}+8A_{\delta_1^2}-48A_{\delta_{01a}}=0\\
A_{\lambda\delta_2}-A_{\delta_1\delta_2}=0\\
2A_{\lambda^2}+24A_{\lambda\delta_0}-2A_{\lambda\delta_1}+288A_{\delta_0^2}-24A_{\delta_0\delta_1}+2A_{\delta_1^2}+144A_{\delta_{00}}+A_{\delta_{1,1}}=0\\
-4A_{\lambda\delta_1}+3A_{\lambda\delta_2}-48A_{\delta_0\delta_1}+8 A_{\delta_1^2}-3A_{\delta_1\delta_2}-12A_{\delta_{01a}}+3A_{\delta_{1,1}}=0\\
8A_{\delta_0^2}-4A_{\delta_0\delta_1}+2A_{\delta_1^2}-2A_{\delta_1\delta_2}+2A_{\delta_2^2}+4A_{\delta_{0,0}}+A_{\delta_{1,1}}=4\\
-4A_{\lambda\delta_0}-96A_{\delta_0^2}+4A_{\delta_0\delta_1}-48A_{\delta_{00}}-A_{\delta_{1,1}}-12A_{\delta_{01a}}=0\\
48A_{\delta_0^2}-4A_{\delta_1^2}+4A_{\kappa_2}=0\\
16A_{\delta_1^2}-2A_{\delta_2^2}+2A_{\kappa_2}+6A_{\delta_{1,1}}=30\\
-2A_{\lambda\delta_0}+A_{\lambda\delta_1}-44A_{\delta_0^2}+12A_{\delta_0\delta_1}-A_{\delta_1^2}+A_{\kappa_2}-12A_{\delta_{00}}+12A_{\delta_{01a}}+A_{\gamma_1}=0\\
A_{\delta_1\delta_2}-A_{\lambda\delta_2}+A_{\delta_2^2}+A_{\kappa_2}+12A_{\delta_{01a}}+12A_{\gamma_1}=0.
\end{eqnarray*}

Next we look at the pull-back to $\Mbar_{2,1}$. The pull-back of the classes $\kappa_2$, $\lambda\delta_0$, $\lambda\delta_1$, $\lambda\delta_2$, $\delta_0^2$, $\delta_1^2$, $\delta_{00}$, $\gamma_1=\theta_1$, $\delta_{01a}=\delta_{0,g-1}$ and $\delta_{1,1}$ have been computed in \S \ref{pullbacktoM2,1}. Moreover
\begin{eqnarray*}
j^*(\lambda^2) & = & \frac{1}{60}([\Delta_{00}]_Q + [(a)]_Q + [(b)]_Q)\\
j^*(\delta_{0}\delta_1) & = & [(a)]_Q+ [(b)]_Q\\
j^*(\delta_1 \delta_2) & = & -\delta_1 \psi\\
				&=& -\frac{1}{12}[(a)]_Q -2[(d)]_Q\\
j^*(\delta_2^2) & = & \psi^2\\
			&=& \frac{1}{120}([\Delta_{00}]_Q +13 [(a)]_Q - [(b)]_Q + 24 [(c)]_Q + 168 [(d)]_Q ).
\end{eqnarray*}
Considering the coefficient of $[\Delta_{00}]_Q$ 
yields the following relation
\begin{eqnarray*}
A_{\delta_{00}}+\frac{5}{3}A_{\delta_0^2}+\frac{1}{6}A_{\lambda\delta_0}-\frac{1}{60}A_{\lambda\delta_2}+\frac{1}{40}A_{\kappa_2}+\frac{1}{60}A_{\lambda^2}+\frac{1}{120}A_{\delta_2^2}=0.
\end{eqnarray*}

\noindent All in all we get $13$ independent relations, and the class of $\Mbar^1_{4,2}$ follows
\begin{eqnarray*}
2 \left[ \Mbar^1_{4,2} \right]_Q &= & 27 \kappa_2 -339 \lambda^2+ 64\lambda\delta_0+90\lambda\delta_1+ 6\lambda\delta_2- \delta_0^2-8\delta_0\delta_1\\
&& {} + 15\delta_1^2 +6\delta_1\delta_2+9\delta_2^2-4\delta_{00}-6\gamma_1+3\delta_{01a}-36\delta_{1,1}.
\end{eqnarray*}

\section*{Acknowledgments}
This work is part of my PhD thesis. I am grateful to my advisor Gavril Farkas for his guidance. During my PhD I have been supported by the DFG Graduierten\-kolleg 870 Berlin-Zurich and the Berlin Mathematical School. Since October 2011 I have been supported by a postdoctoral fellowship at the Leibniz University in Hannover.

\bibliographystyle{alpha}
\bibliography{Biblio.bib}

\begin{thebibliography}{Mum83}

\bibitem[AC96]{MR1486986}
Enrico Arbarello and Maurizio Cornalba.
\newblock Combinatorial and algebro-geometric cohomology classes on the moduli
  spaces of curves.
\newblock {\em J. Algebraic Geom.}, 5(4):705--749, 1996.

\bibitem[Dia85]{MR791679}
Steven Diaz.
\newblock Exceptional {W}eierstrass points and the divisor on moduli space that
  they define.
\newblock {\em Mem. Amer. Math. Soc.}, 56(327):iv+69, 1985.

\bibitem[Edi92]{MR1177306}
Dan Edidin.
\newblock The codimension-two homology of the moduli space of stable curves is
  algebraic.
\newblock {\em Duke Math. J.}, 67(2):241--272, 1992.

\bibitem[Edi93]{MR1185606}
Dan Edidin.
\newblock Brill-{N}oether theory in codimension-two.
\newblock {\em J. Algebraic Geom.}, 2(1):25--67, 1993.

\bibitem[EH86]{MR846932}
David Eisenbud and Joe Harris.
\newblock Limit linear series: basic theory.
\newblock {\em Invent. Math.}, 85(2):337--371, 1986.

\bibitem[EH87]{MR910206}
David Eisenbud and Joe Harris.
\newblock The {K}odaira dimension of the moduli space of curves of genus {$\geq
  23$}.
\newblock {\em Invent. Math.}, 90(2):359--387, 1987.

\bibitem[EH89]{MR985853}
David Eisenbud and Joe Harris.
\newblock Irreducibility of some families of linear series with
  {B}rill-{N}oether number {$-1$}.
\newblock {\em Ann. Sci. \'Ecole Norm. Sup. (4)}, 22(1):33--53, 1989.

\bibitem[Fab88]{faberthesis}
Carel Faber.
\newblock {\em Chow rings of moduli spaces of curves}.
\newblock PhD thesis, Amsterdam, 1988.

\bibitem[Fab89]{MR1023390}
Carel Faber.
\newblock Some results on the codimension-two {C}how group of the moduli space
  of stable curves.
\newblock In {\em Algebraic curves and projective geometry ({T}rento, 1988)},
  volume 1389 of {\em Lecture Notes in Math.}, pages 66--75. Springer, Berlin,
  1989.

\bibitem[Fab90a]{MR1070600}
Carel Faber.
\newblock Chow rings of moduli spaces of curves. {I}. {T}he {C}how ring of
  {$\overline{\mathcal M}_3$}.
\newblock {\em Ann. of Math. (2)}, 132(2):331--419, 1990.

\bibitem[Fab90b]{MR1078265}
Carel Faber.
\newblock Chow rings of moduli spaces of curves. {II}. {S}ome results on the
  {C}how ring of {$\overline{\mathcal M}_4$}.
\newblock {\em Ann. of Math. (2)}, 132(3):421--449, 1990.

\bibitem[Fab99]{MR1722541}
Carel Faber.
\newblock A conjectural description of the tautological ring of the moduli
  space of curves.
\newblock In {\em Moduli of curves and abelian varieties}, Aspects Math., E33,
  pages 109--129. Vieweg, Braunschweig, 1999.

\bibitem[Far09]{MR2574363}
Gavril Farkas.
\newblock The {F}ermat cubic and special {H}urwitz loci in {$\overline{\M}_g$}.
\newblock {\em Bull. Belg. Math. Soc. Simon Stevin}, 16(5, Linear systems and
  subschemes):831--851, 2009.

\bibitem[FP05]{MR2120989}
Carel Faber and Rahul Pandharipande.
\newblock Relative maps and tautological classes.
\newblock {\em J. Eur. Math. Soc. (JEMS)}, 7(1):13--49, 2005.

\bibitem[Ful69]{MR0260752}
William Fulton.
\newblock Hurwitz schemes and irreducibility of moduli of algebraic curves.
\newblock {\em Ann. of Math. (2)}, 90:542--575, 1969.

\bibitem[Ful98]{MR1644323}
William Fulton.
\newblock {\em Intersection theory}, volume~2 of {\em Ergebnisse der Mathematik
  und ihrer Grenzgebiete. 3. Folge. A Series of Modern Surveys in Mathematics
  [Results in Mathematics and Related Areas. 3rd Series. A Series of Modern
  Surveys in Mathematics]}.
\newblock Springer-Verlag, Berlin, second edition, 1998.

\bibitem[Har84]{MR735335}
Joe Harris.
\newblock On the {K}odaira dimension of the moduli space of curves. {II}. {T}he
  even-genus case.
\newblock {\em Invent. Math.}, 75(3):437--466, 1984.

\bibitem[HM82]{MR664324}
Joe Harris and David Mumford.
\newblock On the {K}odaira dimension of the moduli space of curves.
\newblock {\em Invent. Math.}, 67(1):23--88, 1982.
\newblock With an appendix by William Fulton.

\bibitem[HM98]{MR1631825}
Joe Harris and Ian Morrison.
\newblock {\em Moduli of curves}, volume 187 of {\em Graduate Texts in
  Mathematics}.
\newblock Springer-Verlag, New York, 1998.

\bibitem[Log03]{MR1953519}
Adam Logan.
\newblock The {K}odaira dimension of moduli spaces of curves with marked
  points.
\newblock {\em Amer. J. Math.}, 125(1):105--138, 2003.

\bibitem[Mar46]{MR0024182}
Arturo Maroni.
\newblock Le serie lineari speciali sulle curve trigonali.
\newblock {\em Ann. Mat. Pura Appl. (4)}, 25:343--354, 1946.

\bibitem[MS86]{MR882414}
Gerriet Martens and Frank-{O}laf Schreyer.
\newblock Line bundles and syzygies of trigonal curves.
\newblock {\em Abh. Math. Sem. Univ. Hamburg}, 56:169--189, 1986.

\bibitem[Mum77]{MR0450272}
David Mumford.
\newblock Stability of projective varieties.
\newblock {\em Enseignement Math. (2)}, 23(1-2):39--110, 1977.

\bibitem[Mum83]{MR717614}
David Mumford.
\newblock Towards an enumerative geometry of the moduli space of curves.
\newblock In {\em Arithmetic and geometry, {V}ol. {II}}, volume~36 of {\em
  Progr. Math.}, pages 271--328. Birkh\"auser Boston, Boston, MA, 1983.

\bibitem[SF00]{MR1775309}
Zvezdelina~E. Stankova-Frenkel.
\newblock Moduli of trigonal curves.
\newblock {\em J. Algebraic Geom.}, 9(4):607--662, 2000.

\bibitem[Ste98]{MR1618632}
Frauke Steffen.
\newblock A generalized principal ideal theorem with an application to
  {B}rill-{N}oether theory.
\newblock {\em Invent. Math.}, 132(1):73--89, 1998.

\bibitem[Wah12]{wahl}
Nathalie Wahl.
\newblock Homological stability for mapping class groups of surfaces.
\newblock In {\em Handbook of {M}oduli}, volume III, pages 547--583. Advanced
  {L}ectures in {M}athematics 26, 2012.

\end{thebibliography}

\end{document}